\documentclass[12pt,a4paper,twoside]{article}

\usepackage{a4wide, amssymb, amsmath, amsthm, graphics, comment, xspace, enumerate}
\usepackage{graphicx}
\usepackage[a4paper,colorlinks=true,citecolor=blue,urlcolor=blue,linkcolor=blue,bookmarksopen=true,unicode=true,
          pdffitwindow=true]{hyperref}
\usepackage{algorithmic, algorithm}

\input amssym.def
\input amssym

\allowdisplaybreaks

\newcommand{\NN}{\mathbb N}

\newcommand{\CC}{\mathbb C}
\newcommand{\RR}{\mathbb R}
\newcommand{\ZZ}{\mathbb Z}
\newcommand{\ds}{\displaystyle}
\newcommand{\EE}{\mathcal E}
\newcommand{\DD}{\mathcal D}
\newcommand{\SSS}{\mathcal S}

\newcommand{\OO}{\mathcal O}

\newcommand{\supp}{\mathrm{supp\,}}

\newcommand{\Op}{\mathrm{Op}}

\newtheorem{theorem} {Theorem}[section]
\newtheorem{proposition}{Proposition}[section]
\newtheorem{lemma}{Lemma}[section]

\newtheorem{definition}{Definition}[section]
\newtheorem{remark}{Remark}[section]
\newtheorem{example}{Example}[section]
\newcommand{\beq}{\begin{eqnarray}}
\newcommand{\eeq}{\end{eqnarray}}

\newcommand{\beqs}{\begin{eqnarray*}}
\newcommand{\eeqs}{\end{eqnarray*}}
\begin{document}

\title{Anti-Wick and Weyl quantization on ultradistribution spaces}

\author{Stevan Pilipovic, Bojan Prangoski}

\date{}
\maketitle

\begin{abstract}
The connection between the Anti-Wick and Weyl quantization is given for certain class of global symbols, which corresponding pseudodifferential operators act continuously on the space of tempered ultradistributions of Beurling, respectively, of Roumieu type. The largest subspace of ultradistributions is found for which the convolution with the gaussian kernel exist. This gives a way to extend the definition of Anti-Wick quantization for symbols that are not necessarily tempered ultradistributions.
\end{abstract}

\section{Introduction}

The Anti-Wick and the Weyl quantization of global symbols, as well as their connection, in the case of Schwartz distributions was vastly studied during the years (see for example \cite{NR} and \cite{Shubin} for a systematic approach to the theory). The importance in studying the Anti-Wick quantization lies in the facts that real valued symbols give rise to formally self-adjoint operators and positive symbols give rise to positive operators. On the other hand the Weyl quantization is important because it is closely connected with the Wigner transform and also, the Weyl quantization of real valued symbol is formally self-adjoint operator.\\
\indent The results that we give here are related to the global symbol classes defined and studied in \cite{BojanS}, which corresponding operators act continuously on the space of tempered ultradistributions of Beurling, resp. Roumieu type.\\
\indent For a symbol $a$ which is an element of the space of tempered (ultra)distributions, its Anti-Wick quantization is equal to the Weyl quantization of a symbol $b$ that is given as the convolution of $a$ and the gaussian kernel $e^{-|\cdot|^2}$. The purpose of this paper is twofold. In the first part we extend results from \cite{NR} (see also \cite{Shubin}) to ultradistributions. More precisely, we give the connection of Anti-Wick and Weyl quantization for symbols belonging to specific symbol classes developed by one of the authors in \cite{BojanS}. The last two sections are devoted to finding the largest subspace of ultradistributions for which the convolution with the gaussian kernel exist. The answer to this question in the case of Schwartz distributions was already given in \cite{Wagner}. This gives a way to extend the definition of Anti-Wick operators with symbols that are not necessarily tempered ultradistributions. In particular, we prove theorem \ref{trb}, which gives such class of symbols.\\
\indent The paper is organized as follows:\\
\indent \textbf{Section 1} contains some basic facts concerning spaces of ultradistribution.\\
\indent In \textbf{Section 2} we recall important results related to the symbol classes and their corresponding pseudodifferential operators defined and studied in \cite{BojanS}.\\
\indent \textbf{Section 3} is devoted to the connection between the Anti-Wick and Weyl quantization of symbols belonging to the mentioned symbol classes.\\
\indent In \textbf{Section 4} we find the largest subspace of ultradistributions for which the convolution with $e^{s|\cdot|^2}$, $s\in\RR\backslash\{0\}$, exist.\\
\indent In \textbf{Section 5} we extend the definition of Anti-Wick operators for symbols that are not necessarily tempered ultradistributions, by using the results obtained in the previous sections.

\section{Preliminaries}

The sets of natural, integer, positive integer, real and complex numbers are denoted by $\NN$, $\ZZ$, $\ZZ_+$, $\RR$, $\CC$. We use the symbols for $x\in \RR^d$: $\langle x\rangle =(1+|x|^2)^{1/2} $,
$D^{\alpha}= D_1^{\alpha_1}\ldots D_n^{\alpha_d},\quad D_j^
{\alpha_j}={i^{-1}}\partial^{\alpha_j}/{\partial x}^{\alpha_j}$, $\alpha=(\alpha_1,\alpha_2,\ldots,\alpha_d)\in\NN^d$. If $z\in\CC^d$, by $z^2$ we will denote $z^2_1+...+z^2_d$. Note that, if $x\in\RR^d$, $x^2=|x|^2$.\\
\indent Following \cite{Komatsu1}, we denote by $M_{p}$ a sequence of positive numbers $M_0=1$ so that:\\
\indent $(M.1)$ $M_{p}^{2} \leq M_{p-1} M_{p+1}, \; \; p \in\ZZ_+$;\\
\indent $(M.2)$ $\ds M_{p} \leq c_0H^{p} \min_{0\leq q\leq p} \{M_{p-q} M_{q}\}$, $p,q\in \NN$, for some $c_0,H\geq1$;\\
\indent $(M.3)$  $\ds\sum^{\infty}_{p=q+1}   \frac{M_{p-1}}{M_{p}}\leq c_0q \frac{M_{q}}{M_{q+1}}$, $q\in \ZZ_+$,\\
although in some assertions we could assume the weaker ones $(M.2)'$ and $(M.3)'$ (see \cite{Komatsu1}). For a multi-index $\alpha\in\NN^d$, $M_{\alpha}$ will mean $M_{|\alpha|}$, $|\alpha|=\alpha_1+...+\alpha_d$. Recall,  $m_p=M_p/M_{p-1}$, $p\in\ZZ_+$ and the associated function for the sequence $M_{p}$ is defined by
\beqs
M(\rho)=\sup  _{p\in\NN}\log_+   \frac{\rho^{p}}{M_{p}} , \; \; \rho > 0.
\eeqs
It is non-negative, continuous, monotonically increasing function, which vanishes for sufficiently small $\rho>0$ and increases more rapidly then $(\ln \rho)^p$ when $\rho$ tends to infinity, for any $p\in\NN$.\\
\indent Let $U\subseteq\RR^d$ be an open set and $K\subset\subset U$ (we will use always this notation for a compact subset of an open set). Then $\EE^{\{M_p\},h}(K)$ is the space of all $\varphi\in \mathcal{C}^{\infty}(U)$ which satisfy $\ds\sup_{\alpha\in\NN^d}\sup_{x\in K}\frac{|D^{\alpha}\varphi(x)|}{h^{\alpha}M_{\alpha}}<\infty$ and $\DD^{\{M_p\},h}_K$ is the space of all $\varphi\in \mathcal{C}^{\infty}\left(\RR^d\right)$ with supports in $K$, which satisfy $\ds\sup_{\alpha\in\NN^d}\sup_{x\in K}\frac{|D^{\alpha}\varphi(x)|}{h^{\alpha}M_{\alpha}}<\infty$;
$$
\EE^{(M_p)}(U)=\lim_{\substack{\longleftarrow\\ K\subset\subset U}}\lim_{\substack{\longleftarrow\\ h\rightarrow 0}} \EE^{\{M_p\},h}(K),\,\,\,\,
\EE^{\{M_p\}}(U)=\lim_{\substack{\longleftarrow\\ K\subset\subset U}}
\lim_{\substack{\longrightarrow\\ h\rightarrow \infty}} \EE^{\{M_p\},h}(K),
$$
\beqs
\DD^{(M_p)}(U)=\lim_{\substack{\longrightarrow\\ K\subset\subset U}}\lim_{\substack{\longleftarrow\\ h\rightarrow 0}} \DD^{\{M_p\},h}_K,\,\,\,\,
\DD^{\{M_p\}}(U)=\lim_{\substack{\longrightarrow\\ K\subset\subset U}}\lim_{\substack{\longrightarrow\\ h\rightarrow \infty}} \DD^{\{M_p\},h}_K.
\eeqs
The spaces of ultradistributions and ultradistributions with compact support of Beurling and Roumieu type are defined as the strong duals of $\DD^{(M_p)}(U)$ and $\EE^{(M_p)}(U)$, resp. $\DD^{\{M_p\}}(U)$ and $\EE^{\{M_p\}}(U)$. For the properties of these spaces, we refer to \cite{Komatsu1}, \cite{Komatsu2} and \cite{Komatsu3}. In the future we will not emphasize the set $U$ when $U=\RR^d$. Also, the common notation for the symbols $(M_{p})$ and $\{M_{p}\} $ will be *.\\
\indent For $f\in L^{1} $, its Fourier transform is defined by
$(\mathcal{F}f)(\xi ) = \int_{{\RR^d}} e^{-ix\xi}f(x)dx$, $\xi \in {\RR^d}$.

By $\mathfrak{R}$ is denoted a set of positive sequences which monotonically increases to infinity. For $(r_p)\in\mathfrak{R}$, consider the sequence $N_0=1$, $N_p=M_p\prod_{j=1}^{p}r_j$, $p\in\ZZ_+$. One easily sees that this sequence satisfies $(M.1)$ and $(M.3)'$ and its associated function will be denoted by $N_{r_p}(\rho)$, i.e. $\ds N_{r_{p}}(\rho )=\sup_{p\in\NN} \log_+ \frac{\rho^{p }}{M_p\prod_{j=1}^{p}r_j}$, $\rho > 0$. Note, for given $(r_{p})$ and every $k > 0 $ there is $\rho _{0} > 0$ such that $\ds N_{r_{p}} (\rho ) \leq M(k \rho )$, for $\rho > \rho _{0}$. In \cite{Komatsu3} it is proven that for each $K\subset\subset \RR^d$, the topology of $\ds \DD_K^{\{M_p\}}=\lim_{\substack{\longrightarrow\\ h\rightarrow\infty}}\DD^{\{M_p\},h}_K$ is generated by the seminorms $\ds p_{(t_j),K}(\varphi)=\sup_{\alpha\in\NN^d}\frac{\left\|D^{\alpha}\varphi\right\|_{L^{\infty}}} {M_{\alpha}\prod_{j=1}^{|\alpha|}t_j}$, where $(t_j)\in\mathfrak{R}$. In \cite{BojanL} the following lemma is proven.
\begin{lemma}\label{pst}
Let $(k_p)\in\mathfrak{R}$. There exists $(k'_p)\in\mathfrak{R}$ such that $k'_p\leq k_p$ and
\beqs
\prod_{j=1}^{p+q}k'_j\leq 2^{p+q}\prod_{j=1}^{p}k'_j\cdot\prod_{j=1}^{q}k'_j, \mbox{ for all }p,q\in\ZZ_+.
\eeqs
\end{lemma}
\noindent Hence, for every $(k_p)\in\mathfrak{R}$, we can find $(k'_p)\in\mathfrak{R}$, as in lemma \ref{pst}, such that $N_{k_p}(\rho)\leq N_{k'_p}(\rho)$, $\rho>0$ and the sequence $N_0=1$, $N_p=M_p\prod_{j=1}^{p}k'_j$, $p\in\ZZ_+$, satisfies $(M.2)$ if $M_p$ does.\\
\indent From now on, we always assume that $M_p$ satisfies $(M.1)$, $(M.2)$ and $(M.3)$. It is said that $P(\xi ) =\sum _{\alpha \in \NN^d}c_{\alpha } \xi^{\alpha}$, $\xi \in \RR^d$, is an ultrapolynomial of the class $(M_{p})$, resp. $\{M_{p}\}$, whenever the coefficients $c_{\alpha }$ satisfy the estimate $|c_{\alpha }|  \leq C L^{|\alpha| }/M_{\alpha}$, $\alpha \in \NN^d$ for some $L > 0$ and $C>0$, resp. for every $L > 0 $ and some $C_{L} > 0$. The corresponding operator $P(D)=\sum_{\alpha} c_{\alpha}D^{\alpha}$ is an ultradifferential operator of the class $(M_{p})$, resp. $\{M_{p}\}$ and they act continuously on $\EE^{(M_p)}(U)$ and $\DD^{(M_p)}(U)$, resp. $\EE^{\{M_p\}}(U)$ and $\DD^{\{M_p\}}(U)$ and the corresponding spaces of ultradistributions. In \cite{BojanL} a special class of ultrapolynomials of class * were constructed. We summarize the results obtained there in the following proposition.
\begin{proposition}\label{orn}
Let $c>0$ and $k>0$, resp. $c>0$ and $(k_p)\in\mathfrak{R}$ are arbitrary but fixed. Then there exist $l>0$ and $q\in\ZZ_+$, resp. there exist $(l_p)\in\mathfrak{R}$ and $q\in\ZZ_+$ such that $\ds P_l(z)=\prod_{j=q}^{\infty}\left(1+\frac{z^2}{l^2 m_j^2}\right)$, resp. $\ds P_{l_p}(z)=\prod_{j=q}^{\infty}\left(1+\frac{z^2}{l_j^2 m_j^2}\right)$, is an entire function that doesn't have zeroes on the strip $W=\RR^d+i\{y\in\RR^d||y_j|\leq c,\,j=1,...,d\}$. $P_l(x)$, resp. $P_{l_p}(x)$, is an ultrapolynomial of class *. Moreover $|P_l(z)|\geq \tilde{C}e^{M(|z|/k)}$, resp. $|P_{l_p}(z)|\geq \tilde{C}e^{N_{k_p}(|z|)}$, $z\in W$, for some $\tilde{C}>0$ and $\ds\left|\partial^{\alpha}_x\frac{1}{P_l(x)}\right|\leq C\cdot\frac{\alpha!}{r^{|\alpha|}}e^{-M\left(|x|/k\right)}$, resp. $\ds\left|\partial^{\alpha}_x\frac{1}{P_{l_p}(x)}\right|\leq C\cdot\frac{\alpha!}{r^{|\alpha|}}e^{-N_{k_p}(|x|)}$, $x\in\RR^d$, $\alpha\in\NN^d$, where $C$ depends on $k$ and $l$, resp. $(k_p)$ and $(l_p)$, and $M_p$; $r\leq c$ arbitrary but fixed.
\end{proposition}
\indent We denote by $\SSS^{M_{p},m}_{2} \left(\RR^d\right)$, $m > 0$, the space of all smooth functions $\varphi$ which satisfy
\beqs
\sigma_{m,2}(\varphi ): = \left( \sum_{\alpha,\beta\in\NN^d} \int_{\RR^d} \left|\frac{m^{|\alpha|+|\beta|}\langle x\rangle^{|\alpha|}D^{\beta}\varphi(x)}{M_{\alpha}M_{\beta}}\right| ^{2} dx \right) ^{1/2}<\infty,
\eeqs
supplied with the topology induced by the norm $\sigma _{m,2}$. The spaces $\SSS'^{(M_{p})}$ and $\SSS'^{\{M_{p}\}}$ of tempered ultradistributions of Beurling and Roumieu type respectively, are defined as the strong duals of the spaces $\ds\SSS^{(M_{p})}=\lim_{\substack{\longleftarrow\\ m\rightarrow\infty}}\SSS^{M_{p},m}_{2}\left(\RR^d\right)$ and $\ds\SSS^{\{M_{p}\}}=\lim_{\substack{\longrightarrow\\ m\rightarrow 0}}\SSS^{M_{p},m}_{2}\left(\RR^d\right)$, respectively. In \cite{PilipovicK} (see also \cite{PilipovicU}) it is proved that the sequence of norms $\sigma_{m,2}$, $m > 0$, is equivalent with the sequences of norms $\|\cdot\|_{m}$, $m > 0$, where $\ds \|\varphi\|_m:=\sup_{\alpha\in \NN^d}\frac{m^{|\alpha|}\| D^{\alpha}\varphi(\cdot) e^{M(m|\cdot|)}\|_{L_{\infty}}}{M_{\alpha }}$. If we denote by $\SSS^{M_p,m}_{\infty}\left(\RR^d\right)$ the space of all infinitely differentiable functions on $\RR^d$ for which the norm $\|\cdot\|_m$ is finite (obviously it is a Banach space), then $\ds\SSS^{(M_p)}\left(\RR^d\right)=\lim_{\substack{\longleftarrow\\ m\rightarrow\infty}} \SSS^{M_p,m}_{\infty}\left(\RR^d\right)$ and $\ds\SSS^{\{M_p\}}\left(\RR^d\right)=\lim_{\substack{\longrightarrow\\ m\rightarrow 0}} \SSS^{M_p,m}_{\infty}\left(\RR^d\right)$. Also, for $m_2>m_1$, the inclusion $\SSS^{M_p,m_2}_{\infty}\left(\RR^d\right)\longrightarrow\SSS^{M_p,m_1}_{\infty}\left(\RR^d\right)$ is a compact mapping. So, $\SSS^*\left(\RR^d\right)$ is a $(FS)$ - space in $(M_p)$ case, resp. a $(DFS)$ - space in the $\{M_p\}$ case. Moreover, they are nuclear spaces. In \cite{PilipovicK} (see also \cite{PilipovicT}) it is proved that $\ds\SSS^{\{M_{p}\}} = \lim_{\substack{\longleftarrow\\ (r_{i}), (s_{j}) \in \mathfrak{R}}}\SSS^{M_{p}}_{(r_{p}),(s_{q})}$, where $\ds\SSS^{M_{p}}_{(r_{p}),(s_{q})}=\left\{\varphi \in \mathcal{C}^{\infty} \left(\RR^d\right)|\|\varphi\|_{(r_{p}),(s_{q})}<\infty\right\}$ and $\ds\|\varphi\|_{(r_{p}),(s_{q})} =\sup_{\alpha\in \NN^d}\frac{\left\|D^{\alpha}\varphi(x)e^{N_{s_p}(|x|)}\right\|_{L^{\infty}}} {M_{\alpha}\prod^{|\alpha|}_{p=1}r_{p}}$. Also, the Fourier transform is a topological automorphism of $\SSS^*$ and of $\SSS'^*$.\\
\indent Denote by $\OO_C'^*$ the space of convolutors for $\SSS^*$, i.e. the space of all $T\in\SSS'^*$ for which the mapping $\varphi\mapsto T*\varphi$ is well defined and continuous mapping from $\SSS^*$ to itself. Denote by $\OO_M^*$ the space of multipliers for $\SSS^*$, i.e. the space of all $\psi\in\EE^*$ for which the mapping $\varphi\mapsto \psi\varphi$ is well defined and continuous mapping from $\SSS^*$ to itself. For the properties of these spaces we refer to \cite{PBD}.\\
\indent We need the following kernel theorem for $\SSS'^*$ from \cite{BojanS}. The $(M_p)$ case was already considered in \cite{LPK} (the authors used the characterization of Fourier-Hermite coefficients of the elements of the space in the proof of the kernel theorem).

\begin{proposition}\label{ktr}
The following isomorphisms of locally convex spaces hold
\beqs
&{}&\SSS^*\left(\RR^{d_1}\right)\hat{\otimes}\SSS^*\left(\RR^{d_2}\right)\cong\SSS^*\left(\RR^{d_1+d_2}\right)\cong \mathcal{L}_b\left(\SSS'^*\left(\RR^{d_1}\right),\SSS^*\left(\RR^{d_2}\right)\right),\\
&{}&\SSS'^*\left(\RR^{d_1}\right)\hat{\otimes}\SSS'^*\left(\RR^{d_2}\right)\cong\SSS'^*\left(\RR^{d_1+d_2}\right)\cong \mathcal{L}_b\left(\SSS^*\left(\RR^{d_1}\right),\SSS'^*\left(\RR^{d_2}\right)\right).
\eeqs
\end{proposition}

\indent As in \cite{PilipovicT}, we define $\DD^*_{L^{\infty}}$ by $\ds \DD^{(M_p)}_{L^{\infty}}=\lim_{\substack{\longleftarrow\\ h\rightarrow\infty}}\DD^{M_p}_{L^{\infty},h}$, resp. $\ds\DD^{\{M_p\}}_{L^{\infty}}=\lim_{\substack{\longrightarrow\\ h\rightarrow 0}}\DD^{M_p}_{L^{\infty},h}$, where $\DD^{M_p}_{L^{\infty},h}$ is the Banach space of all $\varphi\in\mathcal{C}^{\infty}$ for which the norm $\ds\sup_{\alpha\in\NN^d}\frac{h^{|\alpha|}\left\|D^{\alpha}\varphi\right\|_{L^{\infty}}}{M_{\alpha}}$ is finite. We define $\tilde{\DD}^{\{M_p\}}_{L^{\infty}}$ as the space of all $\mathcal{C}^{\infty}$ functions such that, for every $(t_j)\in\mathfrak{R}$, the norm $\ds p_{(t_j)}(\varphi)=\sup_{\alpha\in\NN^d}\frac{\left\|D^{\alpha}\varphi\right\|_{L^{\infty}}} {M_{\alpha}\prod_{j=1}^{|\alpha|}t_j}$ is finite. The space $\tilde{\DD}^{\{M_p\}}_{L^{\infty}}$ is complete Hausdorff locally convex space because $\ds \tilde{\DD}^{\{M_p\}}_{L^{\infty}}=\lim_{\substack{\longleftarrow\\ (t_j)\in\mathfrak{R}}}\tilde{\DD}^{M_p}_{L^{\infty},(t_j)}$, where $\tilde{\DD}^{M_p}_{L^{\infty},(t_j)}$ is the Banach space of all $\mathcal{C}^{\infty}$ functions for which the norm $p_{(t_j)}(\cdot)$ is finite. In \cite{PilipovicT} it is proved that $\DD^{\{M_p\}}_{L^{\infty}}=\tilde{\DD}^{\{M_p\}}_{L^{\infty}}$ as sets and the former has a stronger topology than the later. Denote by $\dot{\mathcal{B}}^{(M_p)}$, resp. $\dot{\tilde{\mathcal{B}}}^{\{M_p\}}$ the completion of $\DD^{(M_p)}$, resp. $\DD^{\{M_p\}}$, in $\DD^{(M_p)}_{L^{\infty}}$, resp. $\tilde{\DD}^{\{M_p\}}_{L^{\infty}}$. The strong dual of $\dot{\mathcal{B}}^{(M_p)}$, resp. $\dot{\tilde{\mathcal{B}}}^{\{M_p\}}$, will be denoted by $\DD_{L^1}'^{(M_p)}$, resp. $\tilde{\DD}'^{\{M_p\}}_{L^1}$. For the properties of these spaces we refer to \cite{PilipovicT}.

\section{A class of pseudo-differential operators}

In this section we will give a brief overview of the global symbol classes constructed in \cite{BojanS}. It is important to note that similar symbol classes were considered by M. Cappiello in \cite{C3} and \cite{C4}. All the results that we give in this section can be found in \cite{BojanS}.\\
\indent Let $a\in\SSS'^{*}\left(\RR^{2d}\right)$. For $\tau\in\RR$, consider the ultradistribution
\beq\label{3}
K_{\tau}(x,y)=\mathcal{F}^{-1}_{\xi\rightarrow x-y}(a)((1-\tau)x+\tau y,\xi)\in\SSS'^{*}\left(\RR^{2d}\right).
\eeq
Let $\Op_{\tau}(a)$ be the operator from $\SSS^*$ to $\SSS'^{*}$ corresponding to the kernel $K_{\tau}(x,y)$, i.e.
\beq
\langle \Op_{\tau}(a)u,v\rangle=\langle K_{\tau},v\otimes u\rangle,\, u,v\in\SSS^{*}\left(\RR^d\right).
\eeq
$a$ will be called the $\tau$-symbol of the pseudo-differential operator $\Op_{\tau}(a)$. When $\tau=0$, we will denote $\Op_{0}(a)$ by $a(x,D)$. When $a\in\SSS^{*}\left(\RR^{2d}\right)$,
\beq\label{5}
\Op_{\tau}(a)u(x)=\frac{1}{(2\pi)^d}\int_{\RR^{2d}}e^{i(x-y)\xi}a((1-\tau)x+\tau y,\xi)u(y)dyd\xi,
\eeq
where the integral is absolutely convergent.

\begin{proposition}\label{6}
The correspondence $a\mapsto K_{\tau}$ is an isomorphism of $\SSS^{*}\left(\RR^{2d}\right)$, of $\SSS'^{*}\left(\RR^{2d}\right)$ and of $L^2\left(\RR^{2d}\right)$. The inverse map is given by
\beqs
a(x,\xi)=\mathcal{F}_{y\rightarrow\xi}K_{\tau}(x+\tau y,x-(1-\tau)y).
\eeqs
\end{proposition}

Operators with symbols in $\SSS^{*}$ correspond to kernels in $\SSS^{*}$ and by proposition \ref{ktr}, those extend to continuous operators from $\SSS'^{*}$ to $\SSS^{*}$. We will call these *-regularizing operators.\\
\indent Let $A_p$ and $B_p$ be sequences that satisfy $(M.1)$, $(M.3)'$ and $A_0=1$ and $B_0=1$. Moreover, let $A_p\subset M_p$ and $B_p\subset M_p$ i.e. there exist $c_0>0$ and $L>0$ such that $A_p\leq c_0 L^pM_p$ and $B_p\leq c_0 L^pM_p$, for all $p\in\NN$ (it is obvious that without losing generality we can assume that this $c_0$ is the same with $c_0$ from the conditions $(M.2)$ and $(M.3)$ for $M_p$). For $0<\rho\leq 1$, define $\Gamma_{A_p,B_p,\rho}^{M_p,\infty}\left(\RR^{2d};h,m\right)$ as the space of all $a\in \mathcal{C}^{\infty}\left(\RR^{2d}\right)$ for which the following norm is finite
\beqs
\|a\|_{h,m,\Gamma}=\sup_{\alpha,\beta}\sup_{(x,\xi)\in\RR^{2d}}\frac{\left|D^{\alpha}_{\xi}D^{\beta}_x a(x,\xi)\right|
\langle (x,\xi)\rangle^{\rho|\alpha|+\rho|\beta|}e^{-M(m|\xi|)}e^{-M(m|x|)}}{h^{|\alpha|+|\beta|}A_{\alpha}B_{\beta}}.
\eeqs
It is easily verified that it is a Banach space. Define
\beqs
\Gamma_{A_p,B_p,\rho}^{(M_p),\infty}\left(\RR^{2d};m\right)=\lim_{\substack{\longleftarrow\\h\rightarrow 0}}
\Gamma_{A_p,B_p,\rho}^{M_p,\infty}\left(\RR^{2d};h,m\right),\,
\Gamma_{A_p,B_p,\rho}^{(M_p),\infty}\left(\RR^{2d}\right)=\lim_{\substack{\longrightarrow\\m\rightarrow\infty}}
\Gamma_{A_p,B_p,\rho}^{(M_p),\infty}\left(\RR^{2d};m\right),\\
\Gamma_{A_p,B_p,\rho}^{\{M_p\},\infty}\left(\RR^{2d};h\right)=\lim_{\substack{\longleftarrow\\m\rightarrow 0}}
\Gamma_{A_p,B_p,\rho}^{M_p,\infty}\left(\RR^{2d};h,m\right),\,
\Gamma_{A_p,B_p,\rho}^{\{M_p\},\infty}\left(\RR^{2d}\right)=\lim_{\substack{\longrightarrow\\h\rightarrow\infty}}
\Gamma_{A_p,B_p,\rho}^{\{M_p\},\infty}\left(\RR^{2d};h\right).
\eeqs
$\Gamma_{A_p,B_p,\rho}^{(M_p),\infty}\left(\RR^{2d};m\right)$ and $\Gamma_{A_p,B_p,\rho}^{\{M_p\},\infty}\left(\RR^{2d};h\right)$ are $(F)$ - spaces. $\Gamma_{A_p,B_p,\rho}^{(M_p),\infty}\left(\RR^{2d}\right)$ and $\Gamma_{A_p,B_p,\rho}^{\{M_p\},\infty}\left(\RR^{2d}\right)$ are barreled and bornological locally convex spaces.

\begin{theorem}
Let $a\in\Gamma_{A_p,B_p,\rho}^{*,\infty}\left(\RR^{2d}\right)$. Then the integral (\ref{5}) is well defined as an iterated integral. The ultradistribution $\Op_{\tau}(a)u$, $u\in\SSS^*$, coincides with the function defined by that iterated integral.
\end{theorem}

\begin{theorem}\label{npr}
The mapping $(a,u)\mapsto \Op_{\tau}(a)u$, $\Gamma_{A_p,B_p,\rho}^{*,\infty}\left(\RR^{2d}\right)\times
\SSS^*\left(\RR^d\right)\longrightarrow\SSS^*\left(\RR^d\right)$, is hypocontinuous.
\end{theorem}

Let $\ds\rho_1=\inf\{\rho\in\RR_+|A_p\subset M_p^{\rho}\}$ and $\ds\rho_2=\inf\{\rho\in\RR_+|B_p\subset M_p^{\rho}\}$ and put $\rho_0=\max\{\rho_1,\rho_2\}$. Then $0<\rho_0\leq 1$ and for every $\rho$ such that $\rho_0\leq \rho\leq1$, if the larger infimum can be reached, or, otherwise $\rho_0< \rho\leq1$, $A_p\subset M_p^{\rho}$ and $B_p\subset M_p^{\rho}$. So, for every such $\rho$, there exists $c'_0>0$ and $L>0$ (which depend on $\rho$) such that, $A_p\leq c'_0 L^p M_p^{\rho}$, $B_p\leq c'_0 L^p M_p^{\rho}$. Moreover, because $M_p$ tends to infinity, there exists $\tilde{c}>0$ such that $M_p^{\rho}\leq \tilde{c}M_p$, for all such $\rho$. From now on we suppose that $\rho_0\leq\rho\leq 1$, if the larger infimum can be reached, or otherwise $\rho_0<\rho\leq 1$.\\
\indent For $t>0$, put $Q_t=\left\{(x,\xi)\in\RR^{2d}|\langle x\rangle<t, \langle \xi\rangle<t\right\}$ and $Q_t^c=\RR^{2d}\backslash Q_t$. Denote by $FS_{A_p,B_p,\rho}^{M_p,\infty}\left(\RR^{2d};B,h,m\right)$ the vector space of all formal series $\ds \sum_{j=0}^{\infty}a_j(x,\xi)$ such that $a_j\in \mathcal{C}^{\infty}\left(\mathrm{int\,}Q^c_{Bm_j}\right)$, $D^{\alpha}_{\xi} D^{\beta}_x a_j(x,\xi)$ can be extended to continuous function on $Q^c_{Bm_j}$ for all $\alpha,\beta\in\NN^d$ and
\beqs
\sup_{j\in\NN}\sup_{\alpha,\beta}\sup_{(x,\xi)\in Q_{Bm_j}^c}\frac{\left|D^{\alpha}_{\xi}D^{\beta}_x a_j(x,\xi)\right|
\langle (x,\xi)\rangle^{\rho|\alpha|+\rho|\beta|+2j\rho}e^{-M(m|\xi|)}e^{-M(m|x|)}}
{h^{|\alpha|+|\beta|+2j}A_{\alpha}B_{\beta}A_jB_j}<\infty.
\eeqs
In the above, we use the convention $m_0=0$ and hence $Q^c_{Bm_0}=\RR^{2d}$. It is easy to check that $FS_{A_p,B_p,\rho}^{M_p,\infty}\left(\RR^{2d};B,h,m\right)$ is a Banach space. Define
\beqs
FS_{A_p,B_p,\rho}^{(M_p),\infty}\left(\RR^{2d};B,m\right)=\lim_{\substack{\longleftarrow\\h\rightarrow 0}}
FS_{A_p,B_p,\rho}^{M_p,\infty}\left(\RR^{2d};B,h,m\right),\\
FS_{A_p,B_p,\rho}^{(M_p),\infty}\left(\RR^{2d}\right)=\lim_{\substack{\longrightarrow\\B,m\rightarrow\infty}}
FS_{A_p,B_p,\rho}^{(M_p),\infty}\left(\RR^{2d};B,m\right),\\
FS_{A_p,B_p,\rho}^{\{M_p\},\infty}\left(\RR^{2d};B,h\right)=\lim_{\substack{\longleftarrow\\m\rightarrow 0}}
FS_{A_p,B_p,\rho}^{M_p,\infty}\left(\RR^{2d};B,h\right),\\
FS_{A_p,B_p,\rho}^{\{M_p\},\infty}\left(\RR^{2d}\right)=\lim_{\substack{\longrightarrow\\B,h\rightarrow \infty}}
FS_{A_p,B_p,\rho}^{\{M_p\},\infty}\left(\RR^{2d};B,h\right).
\eeqs
$FS_{A_p,B_p,\rho}^{(M_p),\infty}\left(\RR^{2d};B,m\right)$ and $FS_{A_p,B_p,\rho}^{\{M_p\},\infty}\left(\RR^{2d};B,h\right)$ are $(F)$ - spaces. $FS_{A_p,B_p,\rho}^{(M_p),\infty}\left(\RR^{2d}\right)$ and $FS_{A_p,B_p,\rho}^{\{M_p\},\infty}\left(\RR^{2d}\right)$ are barreled and bornological locally convex spaces. Note, also, that the inclusions $\Gamma_{A_p,B_p,\rho}^{*,\infty}\left(\RR^{2d}\right)\longrightarrow FS_{A_p,B_p,\rho}^{*,\infty}\left(\RR^{2d}\right)$, defined as $\ds a\mapsto\sum_{j\in\NN}a_j$, where $a_0=a$ and $a_j=0$, $j\geq 1$, is continuous.

\begin{definition}
Two sums, $\ds\sum_{j\in\NN}a_j,\,\sum_{j\in\NN}b_j\in FS_{A_p,B_p,\rho}^{*,\infty}\left(\RR^{2d}\right)$, are said to be equivalent, in notation $\ds\sum_{j\in\NN}a_j\sim\sum_{j\in\NN}b_j$, if there exist $m>0$ and $B>0$, resp. there exist $h>0$ and $B>0$, such that for every $h>0$, resp. for every $m>0$,
\beqs
\sup_{N\in\ZZ_+}\sup_{\alpha,\beta}\sup_{(x,\xi)\in Q_{Bm_N}^c}\frac{\left|D^{\alpha}_{\xi}D^{\beta}_x \sum_{j<N}\left(a_j(x,\xi)-b_j(x,\xi)\right)\right|
\langle (x,\xi)\rangle^{\rho|\alpha|+\rho|\beta|+2N\rho}}
{h^{|\alpha|+|\beta|+2N}A_{\alpha}B_{\beta}A_NB_N}\cdot\\
\cdot e^{-M(m|\xi|)}e^{-M(m|x|)}<\infty.
\eeqs
\end{definition}

From now on, we assume that $A_p$ and $B_p$ satisfy $(M.2)$. Without losing generality we can assume that the constants $c_0$ and $H$ from the condition $(M.2)$ for $A_p$ and $B_p$ are the same as the corresponding constants for $M_p$.

\begin{theorem}
Let $a\in\Gamma_{A_p,B_p,\rho}^{*,\infty}\left(\RR^{2d}\right)$ be such that $a\sim 0$. Then, for every $\tau\in\RR$, $\Op_{\tau}(a)$ is *-regularizing.
\end{theorem}

\begin{theorem}\label{150}
Let $\ds\sum_{j\in\NN}a_j\in FS_{A_p,B_p,\rho}^{*,\infty}\left(\RR^{2d}\right)$ be given. Than, there exists $a\in\Gamma_{A_p,B_p,\rho}^{*,\infty}\left(\RR^{2d}\right)$, such that $\ds a\sim\sum_{j\in\NN}a_j$.
\end{theorem}
In the proof of the theorem the construction of $a$ is given in the following way. Let $\varphi(x)\in\DD^{(B_p)}\left(\RR^{d}\right)$ and $\psi(\xi)\in\DD^{(A_p)}\left(\RR^{d}\right)$, in the $(M_p)$ case, resp. $\varphi(x)\in\DD^{\{B_p\}}\left(\RR^{d}\right)$ and $\psi(\xi)\in\DD^{\{A_p\}}\left(\RR^{d}\right)$ in the $\{M_p\}$ case, are such that $0\leq\varphi,\psi\leq 1$, $\varphi(x)=1$ when $\langle x\rangle\leq 2$, $\psi(\xi)=1$ when $\langle\xi\rangle\leq 2$ and $\varphi(x)=0$ when $\langle x\rangle\geq 3$, $\psi(\xi)=0$ when $\langle\xi\rangle\geq 3$. Put $\chi(x,\xi)=\varphi(x)\psi(\xi)$, $\ds \chi_n(x,\xi)=\chi\left(\frac{x}{Rm_n},\frac{\xi}{Rm_n}\right)$ for $n\in\ZZ_+$ and $R>0$ and put $\chi_0(x,\xi)=0$. The desired $a$ can be define to be $a(x,\xi)=\sum_j\left(1-\chi_j(x,\xi)\right)a_j(x,\xi)$ for sufficiently large $R$ in the definition of $\chi_n$.

\begin{theorem}\label{200}
Let $\tau,\tau_1\in\RR$ and $a\in\Gamma_{A_p,B_p,\rho}^{*,\infty}\left(\RR^{2d}\right)$. There exists $b\in\Gamma_{A_p,B_p,\rho}^{*,\infty}\left(\RR^{2d}\right)$ and *-regularizing operator $T$ such that $\Op_{\tau_1}(a)=\Op_{\tau}(b)+T$. Moreover,
\beqs
b(x,\xi)\sim\sum_{\beta}\frac{1}{\beta!}(\tau_1-\tau)^{|\beta|}\partial^{\beta}_{\xi}D^{\beta}_xa(x,\xi), \mbox{ in } FS_{A_p,B_p,\rho}^{*,\infty}\left(\RR^{2d}\right).
\eeqs
\end{theorem}

\begin{theorem}
Let $\tau\in\RR$ and $a\in\Gamma_{A_p,B_p,\rho}^{*,\infty}\left(\RR^{2d}\right)$. The transposed operator, ${}^t\Op_{\tau}(a)$, is still a pseudo-differential operator and it is equal to $\Op_{1-\tau}(a(x,-\xi))$. Moreover, there exist $b\in\Gamma_{A_p,B_p,\rho}^{*,\infty}\left(\RR^{2d}\right)$ and *-regularizing operator $T$ such that ${}^t\Op_{\tau}(a)=\Op_{\tau}(b)+T$ and
\beqs
b(x,\xi)\sim\sum_{\alpha}\frac{1}{\alpha!}(1-2\tau)^{|\alpha|}(-\partial_{\xi})^{\alpha}D^{\alpha}_x a(x,-\xi) \mbox{ in } FS_{A_p,B_p,\rho}^{*,\infty}\left(\RR^{2d}\right).
\eeqs
\end{theorem}

\begin{theorem}\label{cef}
Let $a,b\in\Gamma_{A_p,B_p,\rho}^{*,\infty}\left(\RR^{2d}\right)$. There exist $f\in\Gamma_{A_p,B_p,\rho}^{*,\infty}\left(\RR^{2d}\right)$ and *-regularizing operator $T$ such that $a(x,D)b(x,D)=f(x,D)+T$ and $f$ has the asymptotic expansion
\beqs
f(x,\xi)\sim\sum_{\alpha}\frac{1}{\alpha!}\partial^{\alpha}_{\xi}a(x,\xi)D^{\alpha}_x b(x,\xi) \mbox{ in } FS_{A_p,B_p,\rho}^{*,\infty}\left(\RR^{2d}\right).
\eeqs
\end{theorem}

We end this section with the following technical lemma which is also proven in \cite{BojanS}.

\begin{lemma}\label{69}
Let $M_p$ be a sequence which satisfies $(M.1)$, $(M.2)$ and $(M.3)$ and $m$ a positive real. Then, for all $n\in\ZZ_+$, $M(mm_n)\leq 2(c_0m+2)n\ln H+\ln c_0$, where $c_0$ is the constant form the conditions $(M.2)$ and $(M.3)$. If $(t_p)\in\mathfrak{R}$ then, $N_{t_p}(mm_n)\leq n\ln H+\ln c$ for all $n\in\ZZ_+$, where the constant $c$ depends only on $M_p$, $(t_p)$ and $m$, but not on $n$.
\end{lemma}

\section{Anti-Wick quantization}

First we recall definition and some basic facts about the short-time Fourier transform. It will be convenient to introduce some notation to make the definitions less cumbersome. Put $\mathcal{G}_0(x)=\pi^{-d/4}e^{-\frac{1}{2}|x|^2}$ and $\mathcal{G}_{y,\eta}(x)=\pi^{-d/4}e^{ix\eta}e^{-\frac{1}{2}|x-y|^2}$, where $y$ and $\eta$ are parameters in $\RR^d$ and denote by $(\cdot,\cdot)$ the inner product in $L^2$.

\begin{definition}
For $u\in\SSS'^*$ we define the short-time Fourier transform $Vu$ of $u$ as the tempered ultradistribution in $\RR^{2d}$ given by $Vu(y,\eta)=\mathcal{F}_{t\rightarrow \eta}\left(u(t)\mathcal{G}_0(t-y)\right)$.
\end{definition}

\begin{proposition}
The short-time Fourier transform acts continuously $\SSS'^*\left(\RR^d\right)\longrightarrow \SSS'^*\left(\RR^{2d}\right)$, $\SSS^*\left(\RR^d\right)\longrightarrow \SSS^*\left(\RR^{2d}\right)$ and $L^2\left(\RR^d\right)\longrightarrow L^2\left(\RR^{2d}\right)$. Moreover $\|Vu\|_{L^2\left(\RR^{2d}\right)}=(2\pi)^{d/2}\|u\|_{L^2\left(\RR^d\right)}$.
\end{proposition}

\noindent Its adjoint map $V^*:\SSS^*\left(\RR^{2d}\right)\longrightarrow \SSS^*\left(\RR^d\right)$,
\beqs
V^*F(t)=(2\pi)^d\int_{\RR^d}\mathcal{F}^{-1}_{\eta\rightarrow t}
\left(F(y,\eta)\right)\mathcal{G}_0(t-y)dy,\, F\in\SSS^*\left(\RR^{2d}\right)
\eeqs
extends to a well defined and continuous map $\SSS'^*\left(\RR^{2d}\right)\longrightarrow \SSS'^*\left(\RR^d\right)$ and $L^2\left(\RR^{2d}\right)\longrightarrow L^2\left(\RR^d\right)$ and $V^*V=(2\pi)^dI$. Now we can define Anti-Wick operators.

\begin{definition}
Let $a\in\SSS'^*\left(\RR^{2d}\right)$. We define the Anti-Wick operator with symbol $a$ as the map $A_a:\SSS^*\left(\RR^d\right)\longrightarrow\SSS'^*\left(\RR^d\right)$ given by $A_au=(2\pi)^{-d}V^*(aVu)$, $u\in\SSS^*\left(\RR^d\right)$.
\end{definition}
\noindent Observe that, if $a$ is a multiplier for $\SSS^*\left(\RR^{2d}\right)$ (for example an element of $\Gamma_{A_p,B_p,\rho}^{*,\infty}\left(\RR^{2d}\right)$), then $A_a$ maps $\SSS^*\left(\RR^d\right)$ continuously into itself. Also, note that the above formula is equivalent to
\beq\label{240}
\langle A_au,\overline{v}\rangle=(2\pi)^{-d}\langle a,Vu\overline{Vv}\rangle,\quad u,v\in\SSS^*\left(\RR^d\right).
\eeq

\noindent From this, the following propositions follow.

\begin{proposition}
Let $a_n\in\SSS'^*\left(\RR^{2d}\right)$ be a sequence that converges to $a$ in $\SSS'^*\left(\RR^{2d}\right)$, then $A_{a_n} u\longrightarrow A_a u$, for every $u\in\SSS^*\left(\RR^d\right)$.
\end{proposition}

\begin{proposition}
Let $a\in\SSS'^*\left(\RR^{2d}\right)$ be real valued. Then $A_a$ is formally self-adjoint.
\end{proposition}

\noindent If $a$ is locally integrable function of *-ultrapolynomial growth (for example, if it is an element of $\Gamma_{A_p,B_p,\rho}^{*,\infty}\left(\RR^{2d}\right)$), then, by (\ref{240}), we can represent the action of $A_a$ as
\beqs
A_au(x)=\frac{1}{(2\pi)^d}\int_{\RR^{2d}}a(y,\eta)\left(u,\mathcal{G}_{y,\eta}\right)\mathcal{G}_{y,\eta}(x)dyd\eta,
\quad u\in\SSS^*\left(\RR^d\right).
\eeqs

\noindent The proof of the following proposition is the same as in the case of distribution and it will be omitted (see for example \cite{NR}).

\begin{proposition}\label{245}
Let $a\in\SSS'^*\left(\RR^{2d}\right)$. Then $A_a=b^w$ where $b\in\SSS'^*\left(\RR^{2d}\right)$ is given by
\beq\label{247}
b(x,\xi)=\pi^{-d}\left(a(\cdot,\cdot)*e^{-|\cdot|^2-|\cdot|^2}\right)(x,\xi).
\eeq
\end{proposition}

\indent From now on we assume that $A_p=B_p$. Our goal is to represent the Anti-Wick operator $A_a$, for $a\in\Gamma_{A_p,A_p,\rho}^{*,\infty}\left(\RR^{2d}\right)$ as a pseudo-differential operator $b^w$ for some $b\in\Gamma_{A_p,A_p,\rho}^{*,\infty}\left(\RR^{2d}\right)$. First, note that $|\eta|^{2k}\leq k!e^{|\eta|^2}$, for all $k\in\NN$. From this one easily obtains the following inequality
\beq\label{250}
\langle \eta\rangle^{k}\leq 2^{k}\sqrt{k!}e^{|\eta|^2/2}.
\eeq

\begin{theorem}\label{260}
Let $a\in\Gamma_{A_p,A_p,\rho}^{*,\infty}\left(\RR^{2d}\right)$. Then there exists $\tilde{b}\in\Gamma_{A_p,A_p,\rho}^{*,\infty}\left(\RR^{2d}\right)$ and *-regularizing operator $T$ such that $A_a=\tilde{b}^w+T$. Moreover, $\tilde{b}$ has an asymptotic expansion $\sum_j p_j$ in $FS_{A_p,A_p,\rho}^{*,\infty}\left(\RR^{2d}\right)$, where $p_0=a(x,\xi)$ and
\beqs
p_j(x,\xi)=\sum_{2j-1\leq|\alpha+\beta|\leq2j}
\frac{c_{\alpha,\beta}}{\alpha!\beta!}\partial^{\alpha}_{\xi}\partial^{\beta}_x a(x,\xi)
\eeqs
where $\ds c_{\alpha,\beta}=\frac{1}{\pi^d}\int_{\RR^{2d}}\eta^{\alpha}y^{\beta} e^{-|y|^2-|\eta|^2}dyd\eta$.
\end{theorem}
\begin{proof} First we will prove that $\sum_j p_j\in FS_{A_p,A_p,\rho}^{*,\infty}\left(\RR^{2d}\right)$. Note that $c_{\alpha,\beta}=0$ if $|\alpha+\beta|$ is odd. Hence
\beqs
p_j(x,\xi)=\sum_{|\alpha+\beta|=2j}
\frac{c_{\alpha,\beta}}{\alpha!\beta!}\partial^{\alpha}_{\xi}\partial^{\beta}_x a(x,\xi).
\eeqs
If we use the fact $|\eta|^k\leq \sqrt{k!}e^{|\eta|^2/2}$ we have $|c_{\alpha,\beta}|\leq c'\sqrt{|\alpha|!|\beta|!}$, where we put $\ds c'=\frac{1}{\pi^d}\int_{\RR^{2d}}e^{-|y|^2/2-|\eta|^2/2}dyd\eta$. For the derivatives of $p_j$ we have
\beqs
\left|D^{\gamma}_{\xi}D^{\delta}_x p_j(x,\xi)\right|
&\leq&C'_1\sum_{|\alpha+\beta|=2j}\frac{|c_{\alpha,\beta}|}{\alpha!\beta!}\cdot
\frac{h^{|\gamma|+|\delta|+2j}A_{\alpha+\gamma}A_{\beta+\delta}e^{M(m|\xi|)}e^{M(m|x|)}}
{\langle (x,\xi)\rangle^{\rho|\gamma|+\rho|\delta|+2\rho j}}\\
&\leq& C_1\sum_{|\alpha+\beta|=2j}\frac{d^{2j}}{\sqrt{|\alpha|!|\beta|!}}\cdot
\frac{(hH)^{|\gamma|+|\delta|+2j}A_{\gamma}A_{\delta}A_{2j}e^{M(m|\xi|)}e^{M(m|x|)}}
{\langle (x,\xi)\rangle^{\rho|\gamma|+\rho|\delta|+2\rho j}}\\
&\leq&C_2 2^{2j+2d-1}\frac{(hH)^{|\gamma|+|\delta|+2j}(dH)^{2j}A_{\gamma}A_{\delta}A_jA_je^{M(m|\xi|)}e^{M(m|x|)}}
{\langle (x,\xi)\rangle^{\rho|\gamma|+\rho|\delta|+2\rho j}},
\eeqs
i.e., we obtain $\ds\frac{\left|D^{\gamma}_{\xi}D^{\delta}_x p_j(x,\xi)\right|\langle(x,\xi)\rangle^{\rho|\gamma|+\rho|\delta|+2\rho j}
e^{-M(m|\xi|)}e^{-M(m|x|)}}{(2dhH^2)^{|\gamma|+|\delta|+2j}A_{\gamma}A_{\delta}A_jA_j}\leq C$, for all $(x,\xi)\in\RR^{2d}$, $\gamma,\delta\in\NN$, $j\in\NN$. Hence $\sum_j p_j\in FS_{A_p,A_p,\rho}^{*,\infty}\left(\RR^{2d}\right)$. Take $\chi_j$ as in the remark after theorem \ref{150} and define $\tilde{b}=\sum_j (1-\chi_j)p_j$. Then $\tilde{b}\in\Gamma_{A_p,A_p,\rho}^{*,\infty}\left(\RR^{2d}\right)$ and $\tilde{b}\sim\sum_jp_j$. It is enough to prove that $b-\tilde{b}\in\SSS^*$, for $b$ defined as in (\ref{247}). We have
\beqs
b(x,\xi)-\tilde{b}(x,\xi)=\chi_0(x,\xi)b(x,\xi)+\sum_{n=0}^{\infty}\left(\chi_{n+1}-\chi_n\right)
(x,\xi)\left(b(x,\xi)-\sum_{j=0}^n p_j(x,\xi)\right).
\eeqs
By definition, $\chi_0=0$. We Taylor expand $a$ and we obtain
\beqs
a(y,\eta)=\sum_{|\alpha|+|\beta|\leq 2n+1}\frac{1}{\alpha!\beta!}\partial^{\alpha}_{\xi}\partial^{\beta}_x a(x,\xi)(\eta-\xi)^{\alpha}(y-x)^{\beta}+r_{2n+2}(x,y,\xi,\eta),
\eeqs
where $r_{2n+2}$ is the reminder
\beqs
r_{2n+2}(x,y,\xi,\eta)&=&(2n+2)\sum_{|\alpha+\beta|=2n+2}\frac{1}{\alpha!\beta!}(\eta-\xi)^{\alpha}(y-x)^{\beta}\\
&{}&\hspace{30 pt}\cdot\int_0^1(1-t)^{2n+1}\partial^{\alpha}_{\xi}\partial^{\beta}_x a\left(x+t(y-x),\xi+t(\eta-\xi)\right)dt.
\eeqs
If we put this in the expression for $b-\tilde{b}$, keeping in mind the way we defined $p_j$, we obtain
\beqs
b(x,\xi)-\tilde{b}(x,\xi)=\frac{1}{\pi^d}\sum_{n=0}^{\infty}\left(\chi_{n+1}-\chi_n\right)
(x,\xi)\sum_{|\alpha+\beta|=2n+2}\frac{2n+2}{\alpha!\beta!}I_{\alpha,\beta}(x,\xi),
\eeqs
where we put
\beqs
I_{\alpha,\beta}(x,\xi)=\int_0^1\int_{\RR^{2d}}\eta^{\alpha}y^{\beta}
(1-t)^{2n+1}\partial^{\alpha}_{\xi}\partial^{\beta}_x a\left(x+ty,\xi+t\eta\right)e^{-|y|^2-|\eta|^2}dyd\eta dt.
\eeqs
We will estimate the derivatives of $I_{\alpha,\beta}$.\\
$\ds \left|\partial^{\gamma}_{\xi}\partial^{\delta}_x I_{\alpha,\beta}(x,\xi)\right|$
\beqs
&\leq&\int_0^1\int_{\RR^{2d}}|\eta|^{|\alpha|}|y|^{|\beta|}
\left|\partial^{\alpha+\gamma}_{\xi}\partial^{\beta+\delta}_x a\left(x+ty,\xi+t\eta\right)\right|e^{-|y|^2-|\eta|^2}dyd\eta dt\\
&\leq&C'_1\int_0^1\int_{\RR^{2d}}|\eta|^{|\alpha|}|y|^{|\beta|}
\frac{h^{|\gamma|+|\delta|+2n+2}A_{\alpha+\gamma}A_{\beta+\delta}e^{M(m|\xi+t\eta|)}e^{M(m|x+ty|)}}
{\langle (x+ty,\xi+t\eta)\rangle^{\rho|\gamma|+\rho|\delta|+(2n+2)\rho}}e^{-|y|^2-|\eta|^2}dyd\eta dt\\
&\leq&C'_1\int_0^1\int_{\RR^{2d}}
\frac{h^{|\gamma|+|\delta|+2n+2}A_{\gamma+\delta+2n+2}\langle(y,\eta)\rangle^{2n+2}e^{M(m|\xi+t\eta|)}e^{M(m|x+ty|)}}
{\langle (x+ty,\xi+t\eta)\rangle^{(2n+2)\rho}e^{|y|^2+|\eta|^2}}dyd\eta dt\\
&\leq&C''_1\int_0^1\int_{\RR^{2d}}
\frac{(2hL)^{|\gamma|+|\delta|+2n+2}M_{\gamma+\delta+2n+2}^{\rho}\langle(y,\eta)\rangle^{4n+4}
e^{M(m|\xi+t\eta|)}e^{M(m|x+ty|)}}{\langle (x,\xi)\rangle^{(2n+2)\rho}e^{|y|^2+|\eta|^2}}dyd\eta dt\\
&\leq&C_1\frac{\sqrt{(4n+4)!}(8hLH)^{|\gamma|+|\delta|+2n+2}M_{\gamma+\delta}M_{2n+2}^{\rho}}
{\langle (x,\xi)\rangle^{(2n+2)\rho}}
\int_0^1\int_{\RR^{2d}}\frac{e^{M(m|\xi+t\eta|)}e^{M(m|x+ty|)}}{e^{|y|^2/2+|\eta|^2/2}}dyd\eta dt,
\eeqs
where, in the last inequality, we used (\ref{250}). For shorter notations, we will denote the last integral by $\tilde{I}(x,\xi)$. Note that $\langle(x,\xi)\rangle\geq Rm_n$ on the support of $\chi_{n+1}-\chi_n$. For the derivatives of $\left(\chi_{n+1}-\chi_n\right)(x,\xi)I_{\alpha,\beta}(x,\xi)$, we have\\
$\ds \left|\partial^{\gamma}_{\xi}\partial^{\delta}_x \left(\left(\chi_{n+1}-\chi_n\right)
(x,\xi)I_{\alpha,\beta}(x,\xi)\right)\right|$
\beqs
&\leq&\sum_{\substack{\gamma'\leq\gamma\\ \delta'\leq\delta}}{\gamma\choose\gamma'}{\delta\choose\delta'}
\left|\partial^{\gamma-\gamma'}_{\xi}\partial^{\delta-\delta'}_x \left(\left(\chi_{n+1}-\chi_n\right)(x,\xi)\right)\right|\left|\partial^{\gamma'}_{\xi}\partial^{\delta'}_x
I_{\alpha,\beta}(x,\xi)\right|\\
&\leq&C_2\sum_{\substack{\gamma'\leq\gamma\\ \delta'\leq\delta}}{\gamma\choose\gamma'}{\delta\choose\delta'} \frac{h_1^{|\gamma|-|\gamma'|+|\delta|-|\delta'|}A_{\gamma-\gamma'}A_{\delta-\delta'}}
{(Rm_n)^{|\gamma|-|\gamma'|+|\delta|-|\delta'|}}\\
&{}&\hspace{30 pt}\cdot\frac{\sqrt{(4n+4)!}(8hLH)^{|\gamma'|+|\delta'|+2n+2}M_{\gamma'+\delta'}M_{2n+2}^{\rho}}
{(Rm_n)^{(2n+2)\rho}}\cdot\tilde{I}(x,\xi)\\
&\leq&C_3\sum_{\substack{\gamma'\leq\gamma\\ \delta'\leq\delta}}{\gamma\choose\gamma'}{\delta\choose\delta'} \frac{(h_1L)^{|\gamma|-|\gamma'|+|\delta|-|\delta'|}}
{(RM_1)^{|\gamma|-|\gamma'|+|\delta|-|\delta'|}}\\
&{}&\hspace{30 pt}\cdot\frac{\sqrt{(4n+4)!}(8hLH)^{|\gamma'|+|\delta'|+2n+2}H^{2n+2}M_{\gamma+\delta}M_{n+1}^{2\rho}}
{(Rm_n)^{(2n+2)\rho}}\cdot\tilde{I}(x,\xi)\\
&\leq&C_4 \left(\frac{h_1L}{RM_1}+8hLH\right)^{|\gamma|+|\delta|}
\frac{\sqrt{(4n+4)!}(8hLH^3)^{2n+2}M_{\gamma+\delta}M_n^{2\rho}}
{R^{(2n+2)\rho}m_n^{(2n+2)\rho}}\cdot\tilde{I}(x,\xi)\\
&\leq&C'_4 \left(\frac{h_1L}{RM_1}+8hLH\right)^{|\gamma|+|\delta|}
\frac{\sqrt{(4n+4)!}(8hLH^3)^{2n+2}M_{\gamma+\delta}}{R^{(2n+2)\rho}}\cdot\tilde{I}(x,\xi),
\eeqs
where, in the last inequality, we used that
\beqs
m_n^{n+1}\geq m_n\cdot...\cdot m_2\cdot m_1\cdot m_1=M_nM_1.
\eeqs
Let $m'>0$ be arbitrary but fixed. Then one easily proves that $e^{M(m'|(x,\xi)|)}\leq e^{M(m'(|x|+|\xi|))}\leq 2 e^{M(2m'|x|)}e^{M(2m'|\xi|)}$ (one easily proves that $e^{M(\lambda+\nu)}\leq 2e^{M(2\lambda)}e^{M(2\nu)}$). Then we have
\beqs
e^{M(m|\xi+t\eta|)}&=&e^{-M(2m'|\xi|)}e^{M(2m'|\xi|)}e^{M(m|\xi+t\eta|)}\leq 2e^{-M(2m'|\xi|)}e^{M(4m'|t\eta|)}e^{M(4m'|\xi+t\eta|)}e^{M(m|\xi+t\eta|)}\\
&\leq& c_1 e^{-M(2m'|\xi|)}e^{M(4m'|\eta|)}e^{M\left((m+4m')H|\xi+t\eta|\right)},
\eeqs
where, in the last inequality, we used proposition 3.6 of \cite{Komatsu1}. Similarly
\beqs
e^{M(m|x+ty|)}\leq c_1 e^{-M(2m'|x|)}e^{M(4m'|y|)}e^{M\left((m+4m')H|x+ty|\right)}.
\eeqs
Obviously $e^{M(4m'|\eta|)}\leq c_2e^{|\eta|^2/4}$ and $e^{M(4m'|y|)}\leq c_2e^{|y|^2/4}$ for some $c_2>0$ which depends only on $M_p$ and $m'$. We obtain
\beqs
\tilde{I}(x,\xi)\leq c_3e^{-M(m'|(x,\xi)|)}
\int_0^1\left(\int_{\RR^d}\frac{e^{M\left((m+4m')H|x+ty|\right)}}{e^{|y|^2/4}}dy\cdot \int_{\RR^d}\frac{e^{M\left((m+4m')H|\xi+t\eta|\right)}}{e^{|\eta|^2/4}}d\eta\right) dt.
\eeqs
Note that, when $|y|\leq |x|$ we have $e^{M\left((m+4m')H|x+ty|\right)}\leq e^{M\left(2(m+4m')H|x|\right)}\leq e^{M\left(6(m+4m')HRm_{n+1}\right)}$, on the support of $\chi_{n+1}-\chi_n$ (where $|x|\leq 3Rm_{n+1}$). When $|y|>|x|$ we have $e^{M\left((m+4m')H|x+ty|\right)}\leq e^{M\left(2(m+4m')H|y|\right)}\leq c_4e^{|y|^2/8}$, for some $c_4>0$. We obtain\\
$\ds \int_{\RR^d}\frac{e^{M\left((m+4m')H|x+ty|\right)}}{e^{|y|^2/4}}dy$
\beqs
&=&\int_{|y|\leq |x|}\frac{e^{M\left((m+4m')H|x+ty|\right)}}{e^{|y|^2/4}}dy+
\int_{|y|>|x|}\frac{e^{M\left((m+4m')H|x+ty|\right)}}{e^{|y|^2/4}}dy\\
&\leq& e^{M\left(6(m+4m')HRm_{n+1}\right)}\int_{|y|\leq |x|}\frac{1}{e^{|y|^2/4}}dy+
c_4\int_{|y|>|x|}\frac{1}{e^{|y|^2/8}}dy\leq c_5e^{M\left(6(m+4m')HRm_{n+1}\right)}.
\eeqs
We can obtain similar estimate for the other integral. By lemma \ref{69}, we have
\beqs
e^{M\left(6(m+4m')HRm_{n+1}\right)}\leq c_0 H^{2(n+1)\left(6c_0(m+4m')HR+2\right)}.
\eeqs
So, we have
\beqs
\tilde{I}(x,\xi)\leq c_6e^{-M(m'|(x,\xi)|)}e^{2M\left(6(m+4m')HRm_{n+1}\right)}\leq c_7e^{-M(m'|(x,\xi)|)}H^{4(n+1)\left(6c_0(m+4m')HR+2\right)}
\eeqs
on the support of $\chi_{n+1}-\chi_n$. If we insert this in the estimates for the derivatives of the terms $\left(\chi_{n+1}-\chi_n\right)(x,\xi)I_{\alpha,\beta}(x,\xi)$, we obtain\\
$\ds \left|\partial^{\gamma}_{\xi}\partial^{\delta}_x \left(\left(\chi_{n+1}-\chi_n\right)
(x,\xi)I_{\alpha,\beta}(x,\xi)\right)\right|$
\beqs
&\leq&C_5 \left(\frac{h_1L}{RM_1}+8hLH\right)^{|\gamma|+|\delta|}
\frac{\sqrt{(4n+4)!}(8hLH^3)^{2n+2}M_{\gamma+\delta}}{R^{(2n+2)\rho}}\\
&{}&\hspace{130 pt}\cdot e^{-M(m'|(x,\xi)|)}H^{4(n+1)\left(6c_0(m+4m')HR+2\right)}.
\eeqs
We will consider first the $(M_p)$ case. Take $R$ such that $RM_1\geq L$ and $32d/R^{\rho}\leq 1/2$. Then choose $h_1$ such that $h_1\leq 1/(2m')$ and $h$ such that $8hLH^{3+2\left(6c_0(m+4m')HR+2\right)}\leq1$ and $8hLH\leq 1/(2m')$. Note that, the choice of $R$ (and hence $\chi_j$) doesn't depend on $m'$, only on $A_p$, $M_p$ and $a$. For $|\alpha+\beta|=2n+2$, we have
\beqs
\alpha!\beta!\geq\frac{|\alpha|!|\beta|!}{d^{2n+2}}\geq \frac{(2n+2)!}{(2d)^{2n+2}}.
\eeqs
Also, $\sqrt{(4n+4)!}\leq 2^{2n+2}(2n+2)!$. Now we obtain\\
$\ds\sum_{|\alpha+\beta|=2n+2}\frac{2n+2}{\alpha!\beta!}\left|\partial^{\gamma}_{\xi}\partial^{\delta}_x \left(\left(\chi_{n+1}-\chi_n\right)
(x,\xi)I_{\alpha,\beta}(x,\xi)\right)\right|$
\beqs
&\leq&C_5\sum_{|\alpha+\beta|=2n+2}\frac{2^{2n+2}(2d)^{2n+2}}{(2n+2)!}
\left(\frac{h_1L}{RM_1}+8hLH\right)^{|\gamma|+|\delta|}\frac{2^{2n+2}(2n+2)!M_{\gamma+\delta}}{R^{(2n+2)\rho}}
e^{-M(m'|(x,\xi)|)}\\
&\leq&C_5e^{-M(m'|(x,\xi)|)}\frac{M_{\gamma+\delta}}{m'^{|\gamma|+|\delta|}}\cdot
\left(\frac{8d}{R^{\rho}}\right)^{2n+2}\cdot 2^{2n+2+2d-1}
\leq C_6\frac{M_{\gamma+\delta}}{m'^{|\gamma|+|\delta|}}e^{-M(m'|(x,\xi)|)}\cdot\frac{1}{4^{2n+2}},
\eeqs
where, in the last inequality, we put $C_6=2^{2d-1}C_5$. Hence, for the derivatives of
\beqs
\sum_{n=0}^{\infty}\left(\chi_{n+1}-\chi_n\right)
(x,\xi)\sum_{|\alpha+\beta|=2n+2}\frac{2n+2}{\alpha!\beta!}I_{\alpha,\beta}(x,\xi),
\eeqs
we obtain the estimate $\ds C\frac{M_{\gamma+\delta}}{m'^{|\gamma|+|\delta|}}e^{-M(m'|(x,\xi)|)}$ and by the arbitrariness of $m'$, it follows that it is a $\SSS^{(M_p)}$ function. Let us consider the $\{M_p\}$ case. Take $R$ such that $\ds \frac{256dhLH^9}{R^{\rho}}\leq \frac{1}{2}$. Then, choose $m$ and $m'$ such that $6c_0(m+4m')HR\leq 1$. Then we have\\
$\ds \left|\partial^{\gamma}_{\xi}\partial^{\delta}_x \left(\left(\chi_{n+1}-\chi_n\right)
(x,\xi)I_{\alpha,\beta}(x,\xi)\right)\right|$
\beqs
&\leq&C_5 \left(\frac{h_1L}{RM_1}+8hLH\right)^{|\gamma|+|\delta|}
\frac{\sqrt{(4n+4)!}(8hLH^3)^{2n+2}M_{\gamma+\delta}}{R^{(2n+2)\rho}}\\
&{}&\hspace{130 pt}\cdot e^{-M(m'|(x,\xi)|)}H^{4(n+1)\left(6c_0(m+4m')HR+2\right)}\\
&\leq& C_5 \left(\frac{h_1L}{RM_1}+8hLH\right)^{|\gamma|+|\delta|}
\frac{\sqrt{(4n+4)!}(8hLH^3)^{2n+2}M_{\gamma+\delta}}{R^{(2n+2)\rho}}\cdot e^{-M(m'|(x,\xi)|)}H^{12(n+1)}.
\eeqs
So\\
$\ds\sum_{|\alpha+\beta|=2n+2}\frac{2n+2}{\alpha!\beta!}\left|\partial^{\gamma}_{\xi}\partial^{\delta}_x \left(\left(\chi_{n+1}-\chi_n\right)
(x,\xi)I_{\alpha,\beta}(x,\xi)\right)\right|$
\beqs
&\leq&C_5\sum_{|\alpha+\beta|=2n+2}
\left(\frac{h_1L}{RM_1}+8hLH\right)^{|\gamma|+|\delta|}\frac{(8d)^{2n+2}(8hLH^9)^{2n+2}M_{\gamma+\delta}}
{R^{(2n+2)\rho}}e^{-M(m'|(x,\xi)|)}\\
&\leq&C_5e^{-M(m'|(x,\xi)|)}M_{\gamma+\delta}
\left(\frac{h_1L}{RM_1}+8hLH\right)^{|\gamma|+|\delta|}\cdot
\left(\frac{64dhLH^9}{R^{\rho}}\right)^{2n+2}\cdot 2^{2n+2+2d-1}\\
&\leq& C_6M_{\gamma+\delta}\left(\frac{h_1L}{RM_1}+8hLH\right)^{|\gamma|+|\delta|}
e^{-M(m'|(x,\xi)|)}\cdot\frac{1}{4^{2n+2}}.
\eeqs
Hence, for the derivatives of
\beqs
\sum_{n=0}^{\infty}\left(\chi_{n+1}-\chi_n\right)
(x,\xi)\sum_{|\alpha+\beta|=2n+2}\frac{2n+2}{\alpha!\beta!}I_{\alpha,\beta}(x,\xi),
\eeqs
we obtain the estimate $\ds C M_{\gamma+\delta}\frac{1}{m''^{|\gamma|+|\delta|}}e^{-M(m'|(x,\xi)|)}$, where we put $\ds \frac{1}{m''}=\frac{h_1L}{RM_1}+8hLH$, i.e. it is a $\SSS^{\{M_p\}}$ function. In both cases we obtain that $b-\tilde{b}\in\SSS^*$, which completes the proof.
\end{proof}

Now we want to represent the Weyl quantization of $b\in\Gamma_{A_p,A_p,\rho}^{*,\infty}\left(\RR^{2d}\right)$ by an Anti-Wick operator $A_a$, for some $a\in\Gamma_{A_p,A_p,\rho}^{*,\infty}\left(\RR^{2d}\right)$. First we will prove the following technical lemma.

\begin{lemma}
Let $\sum_k q^{(j)}_k\in FS_{A_p,A_p,\rho}^{*,\infty}\left(\RR^{2d}\right)$ for all $j\in\NN$, such that $q^{(j)}_0=...=q^{(j)}_{j-1}=0$. Assume that there exist $m>0$ and $B>0$, resp. $h>0$ and $B>0$, such that $\sum_k q^{(j)}_k\in FS_{A_p,A_p,\rho}^{(M_p),\infty}\left(\RR^{2d};B,m\right)$ for all $j\in\NN$, resp. $\sum_k q^{(j)}_k\in FS_{A_p,A_p,\rho}^{\{M_p\},\infty}\left(\RR^{2d};B,h\right)$ for all $j\in\NN$. Moreover, assume that the constants $C_{j,h}$, resp. $C_{j,m}$, in
\beqs
\sup_{k\in\NN}\sup_{\alpha,\beta}\sup_{(x,\xi)\in Q_{Bm_k}^c}\frac{\left|D^{\alpha}_{\xi}D^{\beta}_x q^{(j)}_k(x,\xi)\right|
\langle (x,\xi)\rangle^{\rho|\alpha|+\rho|\beta|+2k\rho}e^{-M(m|\xi|)}e^{-M(m|x|)}}
{h^{|\alpha|+|\beta|+2k}A_{\alpha}A_{\beta}A_kA_k}=C_{j,h}<\infty
\eeqs
resp. the same with $C_{j,m}$ in place of $C_{j,h}$ in the $\{M_p\}$ case, are bounded for all $j$, i.e. $\ds\sup_j C_{j,h}=C_h<\infty$, resp. $\ds\sup_j C_{j,m}=C_m<\infty$. Then, there exist $p_j\in \mathcal{C}^{\infty}\left(\RR^{2d}\right)$ such that $p_j\sim\sum_k q^{(j)}_k$, for all $j\in\NN$ and $\sum_j p_j\in FS_{A_p,A_p,\rho}^{*,\infty}\left(\RR^{2d}\right)$. Moreover, $\ds\sum_{j=0}^{\infty} p_j\sim\sum_{j=0}^{\infty}\sum_{l=0}^jq^{(l)}_j$ in $FS_{A_p,A_p,\rho}^{*,\infty}\left(\RR^{2d}\right)$.
\end{lemma}
\begin{remark}
$p_j\sim \sum_k q^{(j)}_k$ should be understand as equivalence of the sums $\ds\underbrace{0+...+0}_{j}+p_j+0+...$ and $\sum_k q^{(j)}_k$.
\end{remark}
\begin{proof} Let $R\geq 2B$ and take $p_j$ as in the remark after theorem \ref{150}, i.e. $\ds p_j=\sum_{k=j}^{\infty}(1-\chi_k)q^{(j)}_k$, for $\chi_k$ constructed there. First, we consider the $(M_p)$ case. We will prove that $\sum_j p_j\in FS_{A_p,A_p,\rho}^{(M_p),\infty}\left(\RR^{2d};B,m\right)$, for sufficiently large $R$. Let $h>0$ be arbitrary but fixed. Obviously, without losing generality, we can assume that $h\leq 1$. For simplicity, denote $C_h$ by $C$. Using the fact that $1-\chi_k(x,\xi)=0$ for $(x,\xi)\in Q_{Rm_k}$, we have the estimate\\
$\ds\frac{\left|D^{\alpha}_{\xi}D^{\beta}_x p_j(x,\xi)\right|
\langle (x,\xi)\rangle^{\rho|\alpha|+\rho|\beta|+2\rho j}e^{-M(m|\xi|)}e^{-M(m|x|)}}
{(8hH)^{|\alpha|+|\beta|+2j}A_{\alpha}A_{\beta}A_jA_j}$
\beqs
&\leq&\sum_{k=j}^{\infty}\sum_{\substack{\gamma\leq\alpha\\ \delta\leq\beta}}{\alpha\choose\gamma}{\beta\choose\delta}\left|D^{\alpha-\gamma}_{\xi}D^{\beta-\delta}_x q^{(j)}_k(x,\xi)\right|e^{-M(m|\xi|)}e^{-M(m|x|)}\\
&{}&\hspace{130 pt}\cdot\frac{\left|D^{\gamma}_{\xi}D^{\delta}_x\left(1-\chi_k(x,\xi)\right)\right|
\langle (x,\xi)\rangle^{\rho|\alpha|+\rho|\beta|+2\rho j}}{(8hH)^{|\alpha|+|\beta|+2j}A_{\alpha}A_{\beta}A_jA_j}\\
&\leq&C\sum_{k=j}^{\infty}\sum_{\substack{\gamma\leq\alpha\\ \delta\leq\beta}}{\alpha\choose\gamma}
{\beta\choose\delta}\frac{h^{|\alpha|-|\gamma|+|\beta|-|\delta|+2k}A_{\alpha-\gamma}A_{\beta-\delta}A_kA_k}
{(8hH)^{|\alpha|+|\beta|+2j}A_{\alpha}A_{\beta}A_jA_j}\\
&{}&\hspace{130 pt} \cdot \langle (x,\xi)\rangle^{\rho|\gamma|+\rho|\delta|+2\rho j-2\rho k}
\left|D^{\gamma}_{\xi}D^{\delta}_x\left(1-\chi_k(x,\xi)\right)\right|\\
&\leq&(c_0c'_0)^2C\sum_{k=j}^{\infty}\frac{1}{8^{|\alpha|+|\beta|+2j}H^{2j}}
h^{2(k-j)}H^{2k}L^{2(k-j)}M_{k-j}^{2\rho}\left|1-\chi_k(x,\xi)\right|\langle (x,\xi)\rangle^{2\rho (j-k)}\\
&{}&+(c_0c'_0)^2C\sum_{k=j}^{\infty}\frac{1}{8^{|\alpha|+|\beta|+2j}H^{2j}}\sum_{\substack{\gamma\leq\alpha, \delta\leq\beta\\(\delta,\gamma)\neq(0,0)}}{\alpha\choose\gamma}{\beta\choose\delta}\\
&{}&\hspace{50 pt}\cdot
\frac{h^{2(k-j)}H^{2k}L^{2(k-j)}M_{k-j}^{2\rho}\left|D^{\gamma}_{\xi}D^{\delta}_x\left(1-\chi_k(x,\xi)\right)\right|
\langle (x,\xi)\rangle^{\rho|\gamma|+\rho|\delta|+2\rho j-2\rho k}}{h^{|\gamma|+|\delta|}A_{\gamma}A_{\delta}}\\
&=&S_1+S_2,
\eeqs
where $S_1$ and $S_2$ are the first and the second sum, correspondingly. To estimate $S_1$ note that, on the support of $1-\chi_k$, the inequality $\langle (x,\xi)\rangle\geq Rm_k$ holds. One obtains
\beqs
S_1\leq (c_0c'_0)^2C\sum_{k=j}^{\infty}\frac{(hLH)^{2(k-j)}M_{k-j}^{2\rho}}{R^{2\rho (k-j)}m_k^{2\rho (k-j)}}\leq (c_0c'_0)^2C\sum_{k=0}^{\infty}\frac{(hLH)^{2k}}{R^{2\rho k}}<\infty,
\eeqs
for $R^{\rho}\geq 2LH\geq 2hLH$ (in the second inequality we use the fact that $m_j^j\geq M_j$). For the estimate of $S_2$, note that $D^{\gamma}_{\xi}D^{\delta}_x\left(1-\chi_k(x,\xi)\right)=0$ when $(x,\xi)\in Q_{3Rm_k}^c$, because $(\delta,\gamma)\neq(0,0)$ and $\chi_k(x,\xi)=0$ on $Q_{3Rm_k}^c$. So, for $(x,\xi)\in Q_{3Rm_k}$, we have that $\langle(x,\xi)\rangle\leq \langle x\rangle+\langle \xi\rangle\leq 6Rm_k$. Moreover, from the construction of $\chi$, we have that for the chosen $h$, there exists $C_1>0$ such that $\ds \left|D^{\alpha}_{\xi}D^{\beta}_x \chi(x,\xi)\right|\leq C_1 h^{|\alpha|+|\beta|}A_{\alpha}A_{\beta}$. By using $m_k^k\geq M_k$, one obtains
\beqs
S_2&\leq& (c_0c'_0)^2CC_1\sum_{k=j}^{\infty}\frac{1}{8^{|\alpha|+|\beta|+2j}}\sum_{\substack{\gamma\leq\alpha, \delta\leq\beta\\(\delta,\gamma)\neq(0,0)}}{\alpha\choose\gamma}{\beta\choose\delta}
\frac{(hLH)^{2(k-j)}6^{\rho|\gamma|+\rho|\delta|}M_{k-j}^{2\rho}(Rm_k)^{\rho|\gamma|+\rho|\delta|}}
{R^{2\rho (k-j)}m_k^{2\rho (k-j)}(Rm_k)^{|\gamma|+|\delta|}}\\
&\leq& (c_0c'_0)^2CC_1\sum_{k=0}^{\infty}\frac{(hLH)^{2k}}{R^{2\rho k}},
\eeqs
which is convergent for $R^{\rho}\geq 2LH\geq 2hLH$. Moreover, note that the choice of $R$ for these sums to be convergent does not depend on $j$, hence $\chi_k$ can be chosen to be the same for all $p_j$. So, these estimates does not depend on $j$ and from this it follows that $\sum_j p_j\in FS_{A_p,A_p,\rho}^{(M_p),\infty}\left(\RR^{2d}\right)$ (actually, to be precise, $\sum_j p_j\in FS_{A_p,A_p,\rho}^{(M_p),\infty}\left(\RR^{2d};B,m\right)$, i.e. the same space as for $\sum_k q^{(j)}_k$).\\
\indent In the $\{M_p\}$ case, there exist $h_1,C_1>0$ such that $\ds\left|D^{\alpha}_{\xi}D^{\beta}_x \chi(x,\xi)\right|\leq C_1 h_1^{|\alpha|+|\beta|}A_{\alpha}A_{\beta}$. Arguing in similar fashion, one proves that $\sum_j p_j\in FS_{A_p,A_p,\rho}^{\{M_p\},\infty}\left(\RR^{2d};B,8\tilde{h}H\right)$, where $\tilde{h}=\max\{h,h_1\}$, i.e. $\sum_j p_j\in FS_{A_p,A_p,\rho}^{\{M_p\},\infty}\left(\RR^{2d}\right)$.\\
\indent It remains to prove the second part of the lemma. One easily proves that $\ds\sum_{j=0}^{\infty}\sum_{l=0}^jq^{(l)}_j\in FS_{A_p,A_p,\rho}^{*,\infty}\left(\RR^{2d}\right)$. Note that $\ds\sum_{j=0}^{N-1}p_j-\sum_{j=0}^{N-1}\sum_{l=0}^j q^{(l)}_j=\sum_{j=0}^{N-1}\left(p_j-\sum_{k=j}^{N-1}q^{(j)}_k\right)$. Moreover, for $(x,\xi)\in Q_{3Rm_N}^c$ and $N>j$, $\ds p_j-\sum_{k=j}^{N-1}q^{(j)}_k=\sum_{k=N}^{\infty} \left(1-\chi_k\right)q^{(j)}_k$. This easily follows from the definition of $\chi_k$ and the fact that $m_n$ is monotonically increasing. We will consider first the $(M_p)$ case. For arbitrary but fixed $0<h\leq 1$ and $(x,\xi)\in Q_{3Rm_N}^c$, we estimate as follows\\
$\ds \frac{\left|D^{\alpha}_{\xi}D^{\beta}_x\sum_{k=N}^{\infty}\left(1-\chi_k(x,\xi)\right)q^{(j)}_k(x,\xi)\right|
\langle(x,\xi)\rangle^{\rho|\alpha|+\rho|\beta|+2\rho N}e^{-M(m|\xi|)}e^{-M(m|x|)}}
{(8(1+H)h)^{|\alpha|+|\beta|+2N}A_{\alpha}A_{\beta}A_NA_N}$
\beqs
&\leq&\sum_{k=N}^{\infty}\frac{\left(1-\chi_k(x,\xi)\right)\left|D^{\alpha}_{\xi}D^{\beta}_x q^{(j)}_k(x,\xi)\right|\langle(x,\xi)\rangle^{\rho|\alpha|+\rho|\beta|+2\rho N}e^{-M(m|\xi|)}e^{-M(m|x|)}}
{(8(1+H)h)^{|\alpha|+|\beta|+2N}A_{\alpha}A_{\beta}A_NA_N}\\
&{}&+\sum_{k=N}^{\infty}\sum_{\substack{\gamma\leq\alpha,\delta\leq\beta\\(\delta,\gamma)\neq(0,0)}}
{\alpha\choose\gamma}{\beta\choose\delta}
\left|D^{\alpha-\gamma}_{\xi}D^{\beta-\delta}_x q^{(j)}_k(x,\xi)\right|e^{-M(m|\xi|)}e^{-M(m|x|)}\\
&{}&\hspace{130 pt}\cdot\frac{\left|D^{\gamma}_{\xi}D^{\delta}_x\left(1-\chi_k(x,\xi)\right)\right|
\langle(x,\xi)\rangle^{\rho|\alpha|+\rho|\beta|+2\rho N}}{(8(1+H)h)^{|\alpha|+|\beta|+2N}A_{\alpha}A_{\beta}A_NA_N}\\
&\leq&\frac{C}{64^N}\sum_{k=N}^{\infty}\frac{\left(1-\chi_k(x,\xi)\right)h^{2k-2N}A_kA_k}
{(1+H)^{2N}\langle(x,\xi)\rangle^{2\rho k-2\rho N}A_NA_N}\\
&{}&+\frac{C}{64^N}\sum_{k=N}^{\infty}\frac{1}{8^{|\alpha|+|\beta|}}
\sum_{\substack{\gamma\leq\alpha,\delta\leq\beta\\(\delta,\gamma)\neq(0,0)}}{\alpha\choose\gamma}{\beta\choose\delta}
\frac{h^{2k-2N}\left|D^{\gamma}_{\xi}D^{\delta}_x\left(1-\chi_k(x,\xi)\right)\right|
\langle(x,\xi)\rangle^{\rho|\gamma|+\rho|\delta|}A_kA_k}
{(1+H)^{2N}h^{|\gamma|+|\delta|}\langle(x,\xi)\rangle^{2\rho k-2\rho N}A_{\gamma}A_{\delta}A_NA_N}\\
&=&S_1+S_2,
\eeqs
where $S_1$ and $S_2$ are the first and the second sum, correspondingly. To estimate $S_1$, observe that on the support of $1-\chi_k$ the inequality $\langle(x,\xi)\rangle\geq Rm_k$ holds. Using the monotone increasingness of $m_n$ and $(M.2)$ for $A_p$, one obtains
\beqs
S_1&\leq& \frac{c_0^2C}{64^N}\sum_{k=N}^{\infty}\frac{h^{2k-2N}H^{2k}A_{k-N}A_{k-N}}
{(1+H)^{2N}R^{2\rho k-2\rho N}m_k^{2\rho k-2\rho N}}
\leq \frac{(c_0c'_0)^2C}{64^N}\sum_{k=N}^{\infty}\frac{h^{2k-2N}H^{2k}L^{2k-2N}M_{k-N}^{2\rho}}
{(1+H)^{2N}R^{2\rho k-2\rho N}m_{k-N}^{2\rho k-2\rho N}}\\
&=&\frac{(c_0c'_0)^2C}{64^N}\frac{H^{2N}}{(1+H)^{2N}}\sum_{k=0}^{\infty}\left(\frac{hHL}{R^{\rho}}\right)^{2k}
\leq \frac{(c_0c'_0)^2C}{64^N}\sum_{k=0}^{\infty}\left(\frac{HL}{R^{\rho}}\right)^{2k}=\frac{(c_0c'_0)^2C\tilde{C}}{64^N},
\eeqs
where we put $\ds \tilde{C}=\sum_{k=0}^{\infty}\left(\frac{HL}{R^{\rho}}\right)^{2k}$, for some fixed $R^{\rho}\geq 2HL$. For the sum $S_2$, observe that $D^{\gamma}_{\xi}D^{\delta}_x\left(1-\chi_k(x,\xi)\right)=0$ when $(x,\xi)\in Q_{3Rm_k}^c$, because $(\delta,\gamma)\neq(0,0)$ and $\chi_k(x,\xi)=0$ on $Q_{3Rm_k}^c$. Moreover, from the construction of $\chi$, we have that for the chosen $h$, there exists $C_1>1$ such that $\ds \left|D^{\alpha}_{\xi}D^{\beta}_x \chi(x,\xi)\right|\leq C_1 h^{|\alpha|+|\beta|}A_{\alpha}B_{\beta}$. Now
\beqs
S_2&\leq& \frac{c_0^2CC_1}{64^N}\sum_{k=N}^{\infty}
\frac{1}{8^{|\alpha|+|\beta|}}\sum_{\substack{\gamma\leq\alpha,\delta\leq\beta\\(\delta,\gamma)\neq(0,0)}}
{\alpha\choose\gamma}{\beta\choose\delta}
\frac{h^{2k-2N}6^{|\gamma|+|\delta|}H^{2k}A_{k-N}A_{k-N}}{(1+H)^{2N}R^{2\rho k-2\rho N}m_k^{2\rho k-2\rho N}}\\
&\leq& \frac{c_0^2CC_1}{64^N}\sum_{k=N}^{\infty}\frac{h^{2k-2N}H^{2k}A_{k-N}A_{k-N}}
{(1+H)^{2N}R^{2\rho k-2\rho N}m_k^{2\rho k-2\rho N}}\leq \frac{(c_0c'_0)^2CC_1\tilde{C}}{64^N},
\eeqs
where we used the above estimate for the last sum. So, we have\\
$\ds \frac{\left|D^{\alpha}_{\xi}D^{\beta}_x\sum_{j=0}^{N-1}\left(p_j(x,\xi)-\sum_{k=j}^{N-1}q^{(j)}_k(x,\xi)\right)\right|
\langle(x,\xi)\rangle^{\rho|\alpha|+\rho|\beta|+2\rho N}e^{-M(m|\xi|)}e^{-M(m|x|)}}
{(8(1+H)h)^{|\alpha|+|\beta|+2N}A_{\alpha}A_{\beta}A_NA_N}$
\beqs
&\leq&\sum_{j=0}^{N-1}\frac{\left|D^{\alpha}_{\xi}D^{\beta}_x\left(p_j(x,\xi)-
\sum_{k=j}^{N-1}q^{(j)}_k(x,\xi)\right)\right|
\langle(x,\xi)\rangle^{\rho|\alpha|+\rho|\beta|+2\rho N}e^{-M(m|\xi|)}e^{-M(m|x|)}}
{(8(1+H)h)^{|\alpha|+|\beta|+2N}A_{\alpha}A_{\beta}A_NA_N}\\
&\leq&\sum_{j=0}^{N-1}\frac{2(c_0c'_0)^2CC_1\tilde{C}}{64^N}=\frac{2N(c_0c'_0)^2CC_1\tilde{C}}{64^N},
\eeqs
which is bounded uniformly for all $N\in\ZZ_+$, for $(x,\xi)\in Q_{3Rm_N}^c$, $\alpha,\beta\in\NN^d$. The proof for the $\{M_p\}$ case is similar.
\end{proof}

\begin{theorem}
Let $b\in\Gamma_{A_p,A_p,\rho}^{*,\infty}\left(\RR^{2d}\right)$. There exist $a\in\Gamma_{A_p,A_p,\rho}^{*,\infty}\left(\RR^{2d}\right)$ and *-regularizing operator $T$ such that $b^w=A_a+T$.
\end{theorem}
\begin{proof} Put $p'_{0,0}=b$ and $p'_{k,0}=0$ for all $k\in\ZZ_+$. For $j\in\ZZ_+$, define $p'_{0,j}=...=p'_{j-1,j}=0$ and
\beqs
p'_{k,j}(x,\xi)&=&\sum_{\substack{l_1+l_2+...+l_j=k\\ l_1\geq 1,...,l_j\geq 1}} \sum_{|\alpha^{(1)}+\beta^{(1)}|=2l_1,...,|\alpha^{(j)}+\beta^{(j)}|=2l_j}
\frac{c_{\alpha^{(1)},\beta^{(1)}}\cdot...\cdot c_{\alpha^{(j)},\beta^{(j)}}}{\alpha^{(1)}!\beta^{(1)}!\cdot...\cdot\alpha^{(j)}!\beta^{(j)}!}\\
&{}&\hspace{150 pt}\cdot\partial^{\alpha^{(1)}+...+\alpha^{(j)}}_{\xi}\partial^{\beta^{(1)}+...+\beta^{(j)}}_x b(x,\xi),
\eeqs
for $k\geq j$, $k\in\ZZ_+$. We will prove that $\sum_k p'_{k,j}$ is an element of $FS_{A_p,A_p,\rho}^{*,\infty}\left(\RR^{2d}\right)$. To do this note that, for $k\geq j$,\\
$\ds \left|\partial^{\gamma+\alpha^{(1)}+...+\alpha^{(j)}}_{\xi}\partial^{\delta+\beta^{(1)}+...+\beta^{(j)}}_x b(x,\xi)\right|$
\beqs
&\leq&c_0^2\|b\|_{h,m,\Gamma}\frac{h^{|\gamma|+|\delta|+2k}H^{|\gamma|+|\delta|+2k}A_{\gamma}A_{\delta}A_{2k}
e^{M(m|\xi|)}e^{M(m|x|)}}{\langle(x,\xi)\rangle^{\rho|\gamma|+\rho|\delta|+2\rho k}}\\
&\leq&c_0^3\|b\|_{h,m,\Gamma}\frac{(hH^2)^{|\gamma|+|\delta|+2k}A_{\gamma}A_{\delta}A_kA_k
e^{M(m|\xi|)}e^{M(m|x|)}}{\langle(x,\xi)\rangle^{\rho|\gamma|+\rho|\delta|+2\rho k}}.
\eeqs
If we use the same estimates as in the beginning of the proof of theorem \ref{260}, we have
\beq\label{290}
\frac{|c_{\alpha^{(s)},\beta^{(s)}}|}{\alpha^{(s)}!\beta^{(s)}!}\leq \frac{c'd^{2l_s}}{\sqrt{|\alpha^{(s)}|!|\beta^{(s)}|!}}\leq c'd^{2l_s},
\eeq
for all $s\in\{1,...,j\}$, where $\ds c'=\frac{1}{\pi^d}\int_{\RR^{2d}}e^{-|y|^2/2-|\eta|^2/2}dyd\eta$. Hence
\beqs
\frac{|c_{\alpha^{(1)},\beta^{(1)}}|\cdot...\cdot |c_{\alpha^{(j)},\beta^{(j)}}|}{\alpha^{(1)}!\beta^{(1)}!\cdot...\cdot\alpha^{(j)}!\beta^{(j)}!}\leq c'^jd^{2k}\leq (c'd^2)^k.
\eeqs
The number of ways we can choose the positive integers $l_1,...,l_j$ such that $l_1+...+l_j=k$ is $\ds{{k-1}\choose{j-1}}$. For every fixed $l_1,...,l_j$, we have
\beqs
\sum_{|\alpha^{(s)}+\beta^{(s)}|=2l_s}1={{2l_s+2d-1}\choose{2d-1}}\leq 2^{2l_s+2d-1}=2^{2d-1}4^{l_s},
\eeqs
for $s\in\{1,...,j\}$. So, if we use that $k\geq j$, we have
\beqs
\sum_{\substack{l_1+l_2+...+l_j=k\\ l_1\geq 1,...,l_j\geq 1}} \sum_{|\alpha^{(1)}+\beta^{(1)}|=2l_1,...,|\alpha^{(j)}+\beta^{(j)}|=2l_j}
1\leq 2^{j(2d-1)}4^k{{k-1}\choose{j-1}}\leq 2^{k(2d-1)}4^k2^{k-1}\leq2^{k(2d+2)}.
\eeqs
We obtain
\beqs
\sum_{\substack{l_1+l_2+...+l_j=k\\ l_1\geq 1,...,l_j\geq 1}} \sum_{|\alpha^{(1)}+\beta^{(1)}|=2l_1,...,|\alpha^{(j)}+\beta^{(j)}|=2l_j}
\frac{|c_{\alpha^{(1)},\beta^{(1)}}|\cdot...\cdot |c_{\alpha^{(j)},\beta^{(j)}}|}{\alpha^{(1)}!\beta^{(1)}!\cdot...\cdot\alpha^{(j)}!\beta^{(j)}!}\leq \left(c'2^{2d+2}d^2\right)^k,
\eeqs
i.e.
\beqs
\frac{\left|D^{\gamma}_{\xi}D^{\delta}_x p'_{k,j}(x,\xi)\right|
\langle(x,\xi)\rangle^{\rho|\gamma|+\rho|\delta|+2\rho k}
e^{-M(m|\xi|)}e^{-M(m|x|)}}{\left(c'2^{2d+2}d^2hH^2\right)^{|\gamma|+|\delta|+2k}A_{\gamma}A_{\delta}A_kA_k}\leq c_0^3\|b\|_{h,m,\Gamma},
\eeqs
for all $(x,\xi)\in\RR^{2d}$, $\gamma,\delta\in\NN^d$, $k\in\NN$ (for $k<j$, $p'_{k,j}=0$). So $\sum_k p'_{k,j}\in FS_{A_p,A_p,\rho}^{*,\infty}\left(\RR^{2d}\right)$. Note that $c_0^3\|b\|_{h,m}$ does not depend on $j$, i.e. the estimates are uniform in $j$. By the above lemma, there exist $\mathcal{C}^{\infty}$ functions $b_j$ such that $b_j\sim\sum_k p'_{k,j}$, for $j\in\NN$ and $\sum_j b_j\in FS_{A_p,A_p,\rho}^{*,\infty}\left(\RR^{2d}\right)$. Note that, by the construction in the lemma and the way we define $p'_{k,j}$, $b_0=p'_{0,0}=b$. By theorem \ref{150}, there exists $a\in \Gamma_{A_p,A_p,\rho}^{*,\infty}\left(\RR^{2d}\right)$ such that $a\sim\sum_j (-1)^jb_j$. We will prove that this $a$ satisfies the conditions in the theorem. By theorem \ref{260}, there exist $c\in\Gamma_{A_p,A_p,\rho}^{*,\infty}\left(\RR^{2d}\right)$ and *-regularizing operator $T_1$ such that $A_a=c^w+T_1$ and $\ds c\sim\sum_{j=0}^{\infty}\sum_{|\alpha+\beta|=2j}
\frac{c_{\alpha,\beta}}{\alpha!\beta!}\partial^{\alpha}_{\xi}\partial^{\beta}_x a(x,\xi)$. One obtains
\beqs
c\sim\sum_{j=0}^{\infty}\sum_{l+k=j}\sum_{|\alpha+\beta|=2l}(-1)^k
\frac{c_{\alpha,\beta}}{\alpha!\beta!}\partial^{\alpha}_{\xi}\partial^{\beta}_x b_k(x,\xi).
\eeqs
To prove this, first, by using $\sum_j (-1)^jb_j\in FS_{A_p,A_p,\rho}^{*,\infty}\left(\RR^{2d}\right)$ and (\ref{290}), one easily verifies that the sum is an element of $FS_{A_p,A_p,\rho}^{*,\infty}\left(\RR^{2d}\right)$. Note that\\
$\ds \sum_{j=0}^{N-1}\sum_{|\alpha+\beta|=2j}
\frac{c_{\alpha,\beta}}{\alpha!\beta!}\partial^{\alpha}_{\xi}\partial^{\beta}_x a(x,\xi)-\sum_{j=0}^{N-1}\sum_{l=0}^{j}\sum_{|\alpha+\beta|=2l}(-1)^{j-l}
\frac{c_{\alpha,\beta}}{\alpha!\beta!}\partial^{\alpha}_{\xi}\partial^{\beta}_x b_{j-l}(x,\xi)$
\beqs
&=&\sum_{j=0}^{N-1}\sum_{|\alpha+\beta|=2j}
\frac{c_{\alpha,\beta}}{\alpha!\beta!}\partial^{\alpha}_{\xi}\partial^{\beta}_x a(x,\xi)-\sum_{l=0}^{N-1}\sum_{j=l}^{N-1}\sum_{|\alpha+\beta|=2l}(-1)^{j-l}
\frac{c_{\alpha,\beta}}{\alpha!\beta!}\partial^{\alpha}_{\xi}\partial^{\beta}_x b_{j-l}(x,\xi)\\
&=&\sum_{j=0}^{N-1}\sum_{|\alpha+\beta|=2j}
\frac{c_{\alpha,\beta}}{\alpha!\beta!}\partial^{\alpha}_{\xi}\partial^{\beta}_x a(x,\xi)-\sum_{j=0}^{N-1}\sum_{l=j}^{N-1}\sum_{|\alpha+\beta|=2j}(-1)^{l-j}
\frac{c_{\alpha,\beta}}{\alpha!\beta!}\partial^{\alpha}_{\xi}\partial^{\beta}_x b_{l-j}(x,\xi)\\
&=&\sum_{j=0}^{N-1}\sum_{|\alpha+\beta|=2j}\frac{c_{\alpha,\beta}}{\alpha!\beta!}
\partial^{\alpha}_{\xi}\partial^{\beta}_x\left(a(x,\xi)-\sum_{s=0}^{N-j-1}(-1)^{s}b_{s}(x,\xi)\right).
\eeqs
By using that $a\sim\sum_j (-1)^jb_j$ and the inequality (\ref{290}), one easily proves the desired equivalence. Now, observe that, if we prove the equivalence
\beqs
b\sim\sum_{j=0}^{\infty}\sum_{l+k=j}\sum_{|\alpha+\beta|=2l}(-1)^k
\frac{c_{\alpha,\beta}}{\alpha!\beta!}\partial^{\alpha}_{\xi}\partial^{\beta}_x b_k(x,\xi),
\eeqs
the claim of the theorem will follow. Observe that\\
$\ds \sum_{j=0}^{N-1}\sum_{l+k=j}\sum_{|\alpha+\beta|=2l}(-1)^k
\frac{c_{\alpha,\beta}}{\alpha!\beta!}\partial^{\alpha}_{\xi}\partial^{\beta}_x b_k(x,\xi)-b(x,\xi)$
\beq
&=&\sum_{j=1}^{N-1}\sum_{l+k=j}\sum_{|\alpha+\beta|=2l}(-1)^k
\frac{c_{\alpha,\beta}}{\alpha!\beta!}\partial^{\alpha}_{\xi}\partial^{\beta}_x b_k(x,\xi)\\
&=&\sum_{k=1}^{N-1}(-1)^{k-1}\left(\sum_{j=k}^{N-1}\sum_{|\alpha+\beta|=2(j-k+1)}
\frac{c_{\alpha,\beta}}{\alpha!\beta!}\partial^{\alpha}_{\xi}\partial^{\beta}_x b_{k-1}(x,\xi)-b_k(x,\xi)\right)\label{310}.
\eeq
Because of the way we defined $p'_{s,k}$, for $s\geq k\geq 2$, we have
\beqs
p'_{s,k}(x,\xi)&=&\sum_{l=1}^{s-k+1}\sum_{|\alpha+\beta|=2l}\frac{c_{\alpha,\beta}}{\alpha!\beta!}
\partial^{\alpha}_{\xi}\partial^{\beta}_x\sum_{\substack{l_1+...+l_{k-1}=s-l\\ l_1\geq 1,...,l_{k-1}\geq 1}} \sum_{|\alpha^{(1)}+\beta^{(1)}|=2l_1,...,|\alpha^{(k-1)}+\beta^{(k-1)}|=2l_{k-1}}\\
&{}&\hspace{50 pt}\frac{c_{\alpha^{(1)},\beta^{(1)}}\cdot...\cdot c_{\alpha^{(k-1)},\beta^{(k-1)}}}{\alpha^{(1)}!\beta^{(1)}!\cdot...\cdot\alpha^{(k-1)}!\beta^{(k-1)}!}
\partial^{\alpha^{(1)}+...+\alpha^{(k-1)}}_{\xi}\partial^{\beta^{(1)}+...+\beta^{(k-1)}}_x b(x,\xi)\\
&=&\sum_{l=1}^{s-k+1}\sum_{|\alpha+\beta|=2l}\frac{c_{\alpha,\beta}}{\alpha!\beta!}
\partial^{\alpha}_{\xi}\partial^{\beta}_x p'_{s-l,k-1}(x,\xi).
\eeqs
For $k=1$ one easily checks that the same formula holds for $p'_{s,1}$ (by definition, $p'_{s-l,0}=0$ when $s>l$ and $p'_{0,0}=b$). Hence
\beqs
\sum_{s=k}^{N-1}p'_{s,k}(x,\xi)&=&\sum_{l=1}^{N-k}\sum_{|\alpha+\beta|=2l}\frac{c_{\alpha,\beta}}{\alpha!\beta!}
\partial^{\alpha}_{\xi}\partial^{\beta}_x\left(\sum_{s=l+k-1}^{N-1}p'_{s-l,k-1}(x,\xi)\right)\\
&=&\sum_{l=1}^{N-k}\sum_{|\alpha+\beta|=2l}\frac{c_{\alpha,\beta}}{\alpha!\beta!}
\partial^{\alpha}_{\xi}\partial^{\beta}_x\left(\sum_{s=k-1}^{N-l-1}p'_{s,k-1}(x,\xi)\right).
\eeqs
Now, we obtain\\
$\ds \sum_{j=k}^{N-1}\sum_{|\alpha+\beta|=2(j-k+1)}
\frac{c_{\alpha,\beta}}{\alpha!\beta!}\partial^{\alpha}_{\xi}\partial^{\beta}_x b_{k-1}(x,\xi)-b_k(x,\xi)$
\beqs
&=&\sum_{l=1}^{N-k}\sum_{|\alpha+\beta|=2l}
\frac{c_{\alpha,\beta}}{\alpha!\beta!}\partial^{\alpha}_{\xi}\partial^{\beta}_x b_{k-1}(x,\xi)-\sum_{s=k}^{N-1}p'_{s,k}(x,\xi)+\sum_{s=k}^{N-1}p'_{s,k}(x,\xi)-b_k(x,\xi)\\
&=&\sum_{l=1}^{N-k}\sum_{|\alpha+\beta|=2l}
\frac{c_{\alpha,\beta}}{\alpha!\beta!}\partial^{\alpha}_{\xi}\partial^{\beta}_x \left(b_{k-1}(x,\xi)-\sum_{s=k-1}^{N-l-1}p'_{s,k-1}(x,\xi)\right)+\sum_{s=k}^{N-1}p'_{s,k}(x,\xi)-b_k(x,\xi).
\eeqs
By construction, $b_{k-1}\sim\underbrace{0+...+0}_{k-1}+\sum_{s=k-1}^{\infty}p'_{s,k-1}$. Moreover, by the above estimates for the derivatives of $p'_{s,k}$, the above lemma and its prove it follows that there exist $B>0$, $m>0$ and $\tilde{C}_h>0$ in the $(M_p)$ case, resp. there exist $B>0$, $h>0$ and $\tilde{C}_m>0$ in the $\{M_p\}$ case, such that for every $h>0$
\beqs
\frac{\left|D^{\alpha}_{\xi}D^{\beta}_x\left(b_{k}(x,\xi)- \sum_{s<N}p'_{s,k}(x,\xi)\right)\right|\langle (x,\xi)\rangle^{\rho|\alpha|+\rho|\beta|+2N\rho} e^{-M(m|\xi|)}e^{-M(m|x|)}}{h^{|\alpha|+|\beta|+2N}A_{\alpha}A_{\beta}A_NA_N}\leq \tilde{C}_h,
\eeqs
for all $(x,\xi)\in Q^c_{Bm_N}$, $\alpha,\beta\in\NN^d$ and $k,N\in\NN$, $N>k$, in the $(M_p)$ case, resp. the same as above but for some $h$ and every $m$ with $\tilde{C}_m$ in place of $\tilde{C}_h$, in the $\{M_p\}$ case. Now, if we use the estimate (\ref{290}), we get that\\
$\ds \left|\sum_{|\alpha+\beta|=2l}\frac{c_{\alpha,\beta}}{\alpha!\beta!}
\partial^{\alpha+\gamma}_{\xi}\partial^{\beta+\delta}_x \left(b_{k-1}(x,\xi)-\sum_{s=k-1}^{N-l-1}p'_{s,k-1}(x,\xi)\right)\right|$
\beqs
&\leq&\tilde{C}\sum_{|\alpha+\beta|=2l}\frac{|c_{\alpha,\beta}|}{\alpha!\beta!}
\frac{h^{|\gamma|+|\delta|+2N}A_{\alpha+\gamma}A_{\beta+\delta}
A_{N-l}A_{N-l}e^{M(m|\xi|)}e^{M(m|x|)}}{\langle (x,\xi)\rangle^{\rho|\gamma|+\rho|\delta|+2N\rho}}\\
&\leq&c_0^3\tilde{C}c'd^{2l}\sum_{|\alpha+\beta|=2l}\frac{(hH^2)^{|\gamma|+|\delta|+2N}A_{\gamma}A_{\delta}
A_{N}A_{N}e^{M(m|\xi|)}e^{M(m|x|)}}{\langle (x,\xi)\rangle^{\rho|\gamma|+\rho|\delta|+2N\rho}}\\
&\leq&c_0^3c'\tilde{C}2^{2d-1}\frac{(2hdH^2)^{|\gamma|+|\delta|+2N}A_{\gamma}A_{\delta}
A_{N}A_{N}e^{M(m|\xi|)}e^{M(m|x|)}}{\langle (x,\xi)\rangle^{\rho|\gamma|+\rho|\delta|+2N\rho}},
\eeqs
for all $(x,\xi)\in Q^c_{Bm_N}$, $\gamma,\delta\in\NN^d$, $N\geq l+1$ (in the last inequality we used $\ds\sum_{|\alpha+\beta|=2l}1\leq 2^{2l+2d-1}$), where we put $\tilde{C}=\tilde{C}_h$ in the $(M_p)$ case, resp. $\tilde{C}=\tilde{C}_m$ in the $\{M_p\}$ case. Note that the estimates are uniform in $l$ and $k$. One obtains\\
$\ds\left|\partial^{\gamma}_{\xi}\partial^{\delta}_x\left(\sum_{l=1}^{N-k}\sum_{|\alpha+\beta|=2l}
\frac{c_{\alpha,\beta}}{\alpha!\beta!}\partial^{\alpha}_{\xi}\partial^{\beta}_x \left(b_{k-1}(x,\xi)-\sum_{s=k-1}^{N-l-1}p'_{s,k-1}(x,\xi)\right)\right)\right|$
\beqs
\leq c_0^3c'\tilde{C}2^{2d-1}\frac{(4hdH^2)^{|\gamma|+|\delta|+2N}A_{\gamma}A_{\delta}
A_{N}A_{N}e^{M(m|\xi|)}e^{M(m|x|)}}{\langle (x,\xi)\rangle^{\rho|\gamma|+\rho|\delta|+2N\rho}},
\eeqs
for all $(x,\xi)\in Q^c_{Bm_N}$, $\gamma,\delta\in\NN^d$, $N>k$, with uniform estimates in $k$. Similar estimates hold for $\ds\sum_{s=k}^{N-1}p'_{s,k}(x,\xi)-b_k(x,\xi)$ (by the definition of $b_k$). By using the equality (\ref{310}), we obtain the desired result.
\end{proof}

The importance in the study of the Anti-Wick quantization lies in the following results. The proofs are similar to the case of Schwartz distributions and we omit them (see for example \cite{NR}).

\begin{proposition}
Let $a$ be a locally integrable function with *-ultrapolynomial growth (for example, an element of $\Gamma_{A_p,B_p,\rho}^{*,\infty}\left(\RR^{2d}\right)$). If $a(x,\xi)\geq 0$ for almost every $(x,\xi)\in\RR^{2d}$, then $(A_au,u)_{L^2}\geq 0$, $\forall u\in\SSS^*$. Moreover, if $a(x,\xi)>0$ for almost every $(x,\xi)\in\RR^{2d}$, then $(A_au,u)_{L^2}> 0$, $\forall u\in\SSS^*$, $u\neq 0$.
\end{proposition}

\noindent Nontrivial symbols $a$ that satisfy the conditions of this proposition, for example, are the ultrapolynomials of the form $\sum_{\alpha}c_{2\alpha}\xi^{2\alpha}$, where $c_{2\alpha}>0$ satisfy the necessary conditions for this to be an ultrapolynomial, i.e. there exist $C>0$ and $\tilde{L}>0$, resp. for every $\tilde{L}>0$ there exists $C>0$, such that $|c_{2\alpha}|\leq C\tilde{L}^{2|\alpha|}/M_{2\alpha}$, for all $\alpha\in\NN^d$.

\begin{proposition}
Let $a\in L^{\infty}\left(\RR^{2d}\right)$. Then $A_a$ extends to a bounded operator on $L^2$, with the following estimate of its norm $\|A_a\|_{\mathcal{L}_b\left(L^2\left(\RR^d\right)\right)}\leq \|a\|_{L^{\infty}\left(\RR^{2d}\right)}$.
\end{proposition}

\section{Convolution with the gaussian kernel}

The existence of the convolution of two ultradistributions was studied in \cite{PilipovicC} and \cite{PK} in the Beurling case and in \cite{PB} in the Roumieu case. The convolution of two ultradistributions $S,T\in\DD'^*$ exists if for every $\varphi\in\DD^*$, $(S\otimes T)\varphi^{\Delta}\in\DD_{L^1}'^{(M_p)}\left(\RR^{2d}\right)$, resp. $(S\otimes T)\varphi^{\Delta}\in \tilde{\DD}_{L^1}'^{\{M_p\}}\left(\RR^{2d}\right)$, where $\varphi^{\Delta}(x,y)=\varphi(x+y)$. In that case $S* T$ is defined by $\langle S* T,\varphi\rangle=\langle (S\otimes T)\varphi^{\Delta}, 1\rangle$. We will briefly comment on the meaning of $\langle (S\otimes T)\varphi^{\Delta},1\rangle$ (for the complete theory of the existence of convolution as well as other equivalent definitions, we refer to \cite{PilipovicC} and \cite{PK} for the Beurling case and \cite{PB} for the Roumieu case). In \cite{PilipovicC}, for the Beurling case and \cite{PB} for the Roumieu case, alternative Hausdorff locally convex topology is introduced on $\DD_{L^{\infty}}^{(M_p)}$, resp $\tilde{\DD}_{L^{\infty}}^{\{M_p\}}$, which is weaker than the original topology, stronger than the induced one from $\EE^*$ and $\DD^*$ is continuously and densely injected in it. Moreover, the duals of these spaces with these topologies coincide with $\DD_{L^1}'^{(M_p)}$, resp. with $\tilde{\DD}_{L^1}'^{\{M_p\}}$ as sets. The meaning of $\langle (S\otimes T)\varphi^{\Delta},1\rangle$ is in the sense of these dualities. If $\psi\in\DD^*$ is such that $0\leq \psi \leq1$, $\psi(x)=1$ when $|x|\leq 1$ and $\psi(x)=0$ when $|x|>2$, then $\psi_j\longrightarrow 1$, when $j\longrightarrow \infty$, in the alternative topology of $\DD_{L^{\infty}}^{(M_p)}$, resp. $\tilde{\DD}_{L^{\infty}}^{\{M_p\}}$. So, if $G\in\DD_{L^1}'^{(M_p)}$, resp. $G\in\tilde{\DD}_{L^1}'^{\{M_p\}}$, $\langle G,\psi_j\rangle\longrightarrow \langle G,1\rangle$, when $j\longrightarrow\infty$.\\
\indent Our goal in this section is to find the largest subspace of $\DD'^*$ such that the convolution of each element of that subspace with $e^{s|\cdot|^2}$ exists, where $s\in\RR$, $s\neq0$ is fixed. The general idea is similar to that in \cite{Wagner}, where the case of Schwartz distributions is considered. We will need the following results, concerning the Laplace transform, from \cite{BojanL}.

For a set $B\subseteq\RR^d$ denote by $\mathrm{ch\,}B$ the convex hull of $B$.
\begin{theorem}\label{t1}
Let $B$ be a connected open set in $\RR^d_{\xi}$ and $T\in\DD'^{*}(\RR^d_x)$ be such that, for all $\xi\in B$, $e^{-x\xi}T(x)\in\SSS'^{*}(\RR^d_x)$. Then the Fourier transform $\mathcal{F}_{x\rightarrow\eta}\left(e^{-x\xi}T(x)\right)$ is an analytic function of $\zeta=\xi+i\eta$ for $\xi\in \mathrm{ch\,}B$, $\eta\in\RR^d$. Furthermore, it satisfies the following estimates: for every $K\subset\subset\mathrm{ch\,}B$ there exist $k>0$ and $C>0$, resp. for every $k>0$ there exists $C>0$, such that
\beq\label{300}
|\mathcal{F}_{x\rightarrow\eta}(e^{-x\xi}T(x))(\xi+i\eta)|\leq Ce^{M(k|\eta|)},\, \forall \xi\in K, \forall\eta\in\RR^d.
\eeq
\end{theorem}

\begin{remark}
If, for $S\in\DD'^{*}$, the conditions of the theorem are fulfilled, we call $\mathcal{F}_{x\rightarrow\eta}\left(e^{-x\xi}S(x)\right)$ the Laplace transform of $S$ and denote it by $\mathcal{L}(S)$. Moreover,
\beq\label{25}
\mathcal{L}(S)(\zeta)=\left\langle e^{\varepsilon\sqrt{1+|x|^2}}e^{-x\zeta}S(x),e^{-\varepsilon\sqrt{1+|x|^2}}\right\rangle,
\eeq
for $\zeta\in U+i\RR^d_{\eta}$, where $\overline{U}\subset\subset \mathrm{ch\,}B$ and $\varepsilon$ depends on $U$.\\
\indent If for $S\in\DD'^{*}$ the conditions of the theorem are fulfilled for $B=\RR^d$, then the choice of $\varepsilon$ can be made uniform for all $K\subset\subset\RR^d$.
\end{remark}

\begin{theorem}\label{t2}
Let $B$ be a connected open set in $\RR^d_{\xi}$ and $f$ an analytic function on $B+i\RR^d_{\eta}$. Let $f$ satisfies the condition: for every compact subset $K$ of $B$ there exist $C>0$ and $k>0$, resp. for every $k>0$ there exists $C>0$, such that
\beq\label{n1}
|f(\xi+i\eta)|\leq C e^{M(k|\eta|)},\, \forall\xi\in K, \forall\eta\in\RR^d.
\eeq
Then, there exists $S\in\DD'^{*}(\RR^d_x)$ such that $e^{-x\xi}S(x)\in\SSS'^{*}(\RR^d_x)$, for all $\xi\in B$ and
\beq\label{n4}
\mathcal{L}(S)(\xi+i\eta)=\mathcal{F}_{x\rightarrow\eta}\left(e^{-x\xi}S(x)\right)(\xi+i\eta)=f(\xi+i\eta),\,\, \xi\in B,\, \eta\in\RR^d.
\eeq
\end{theorem}

Put $B^{*}=\{S\in\DD'^{*}|\cosh(k|x|)S\in\SSS'^{*},\, \forall k\geq 0\}$ and for $s\in\RR\backslash\{0\}$, put $B^{*}_s=e^{-s|x|^2}B^{*}$. Obviously $B^{*}\subseteq\SSS'^{*}$ and $B^{*}_s\subseteq\DD'^{*}$. Define
\beqs
A^{*}=\left\{f\in\mathcal{O}\left(\CC^d\right)|\forall K\subset\subset\RR^d_{\xi},\, \exists h,C>0, \mbox{ resp. }\forall h>0,\, \exists C>0,\,\mbox{such that }\right.\\
\left.|f(\xi+i\eta)|\leq C e^{M(h|\eta|)},\, \forall \xi\in K, \forall \eta\in\RR^d\right\},
\eeqs
$A^{*}_{\small\mbox{real}}=\{f_{|\RR^d}|f\in A^{*}\}$ and $A^{*}_s=e^{s|x|^2}A^{*}_{\small\mbox{real}}$. Assume that $k>0$. First we will prove that $\cosh(k|x|)\in \mathcal{C}^{\infty}(\RR^d)$. For $\rho\geq 0$, we have
\beqs
\cosh(k\rho)=\frac{1}{2}\left(\sum_{n=0}^{\infty}\frac{k^n\rho^n}{n!}+
\sum_{n=0}^{\infty}\frac{(-1)^n k^n\rho^n}{n!}\right)=\sum_{n=0}^{\infty}\frac{k^{2n}\rho^{2n}}{(2n)!},
\eeqs
hence $\ds\cosh(k|x|)=\sum_{n=0}^{\infty}\frac{k^{2n}|x|^{2n}}{(2n)!}$ and the function $\ds\sum_{n=0}^{\infty}\frac{k^{2n}|x|^{2n}}{(2n)!}$ is obviously in $\mathcal{C}^{\infty}\left(\RR^d\right)$. We will give another two equivalent definitions of $B^{*}$. We need the following lemmas.

\begin{lemma}\label{27}
Let $k>0$. The function $\ds\frac{\cosh(k|x|)}{\cosh(2k|x|)}$ is an element of $\SSS^*$.
\end{lemma}
\begin{proof} Consider the function $\ds g_k(z)=\sum_{n=0}^{\infty}\frac{k^{2n}(z^2)^n}{(2n)!}$. Obviously $g_k(z)$ is an entire function. Put $W=\{z=x+iy\in\CC^d||x|>2|y|\}$ and consider the set $W_r=W\backslash\overline{B(0,r)}$, where $B(0,r)$ is the ball in $\CC^d$ with center at $0$ and radius $r>0$. Then $\ds\frac{e^{k\sqrt{z^2}}+e^{-k\sqrt{z^2}}}{2}$ is analytic and single valued function on $W_r$, where we take the principal branch of the square root which is analytic on $\CC\backslash (-\infty,0]$. Also, for $z\in W_r$, put $\rho=\sqrt{\left(|x|^2-|y|^2\right)^2+4(xy)^2}$, $\ds\cos\theta= \frac{|x|^2-|y|^2}{\sqrt{\left(|x|^2-|y|^2\right)^2+4(xy)^2}}$ and $\ds\sin\theta=\frac{2xy}{\sqrt{\left(|x|^2-|y|^2\right)^2+4(xy)^2}}$, where $\theta\in (-\pi,\pi)$, from what it follows $\theta\in(-\pi/2,\pi/2)$ (because $\cos\theta>0$ and $\theta\in (-\pi,\pi)$). We will need sharper estimate for $\cos\theta$.
\beqs
\cos\theta&=&\frac{|x|^2-|y|^2}{\sqrt{\left(|x|^2-|y|^2\right)^2+4(xy)^2}}=
\left(1+\left(\frac{2|xy|}{|x|^2-|y|^2}\right)^2\right)^{-1/2}\\
&\geq&\left(1+\left(\frac{|x|^2+|y|^2}{|x|^2-|y|^2}\right)^2\right)^{-1/2}
\geq\left(1+\left(\frac{\frac{5}{4}|x|^2}{\frac{3}{4}|x|^2}\right)^2\right)^{-1/2}=
\frac{3}{\sqrt{34}}.
\eeqs
Then
\beqs
\left|e^{k\sqrt{z^2}}+e^{-k\sqrt{z^2}}\right|&\geq& \left|e^{k\sqrt{z^2}}\right|-\left|e^{-k\sqrt{z^2}}\right|=e^{k\mathrm{Re\,}\sqrt{\rho(\cos\theta+i\sin\theta)}}-
e^{-k\mathrm{Re\,}\sqrt{\rho(\cos\theta+i\sin\theta)}}\\
&=&e^{k\mathrm{Re\,}\sqrt{\rho}\left(\cos\frac{\theta}{2}+i\sin\frac{\theta}{2}\right)}-
e^{-k\mathrm{Re\,}\sqrt{\rho}\left(\cos\frac{\theta}{2}+i\sin\frac{\theta}{2}\right)}\geq
e^{k\sqrt{\rho}\cos\frac{\theta}{2}}-1
\eeqs
where the second equality follows from the fact that we take the principal branch of the square root. Now, using the above estimate for $\cos\theta$, we have
\beqs
\sqrt{\rho}\cos\frac{\theta}{2}=\sqrt{\rho}\sqrt{\frac{\cos\theta+1}{2}}\geq
\sqrt{\rho}\sqrt{\frac{3+\sqrt{34}}{2\sqrt{34}}}.
\eeqs
So, if we put $\ds c_1=\sqrt{\frac{3+\sqrt{34}}{2\sqrt{34}}}$, we obtain
\beq\label{30}
\left|e^{k\sqrt{z^2}}+e^{-k\sqrt{z^2}}\right|\geq e^{c_1 k\sqrt[4]{\left(|x|^2-|y|^2\right)^2+4(xy)^2}}-1\geq
e^{c_1 k\sqrt{|x|^2-|y|^2}}-1>0.
\eeq
Hence $e^{k\sqrt{z^2}}+e^{-k\sqrt{z^2}}$ doesn't have zeroes in $W_r$. Now, $\ds f(z)=\frac{e^{k\sqrt{z^2}}+e^{-k\sqrt{z^2}}}{e^{2k\sqrt{z^2}}+e^{-2k\sqrt{z^2}}}$ is an analytic function on $W_r$. Moreover, because $\left(e^{k\sqrt{z^2}}+e^{-k\sqrt{z^2}}\right)/2=g_k(z)$, for $z\in W_r\cap\RR^d_x$ and from the uniqueness of analytic continuation, it follows $\left(e^{k\sqrt{z^2}}+e^{-k\sqrt{z^2}}\right)/2=g_k(z)$ on $W_r$. Hence $f(z)=g_k(z)/g_{2k}(z)$ on $W_r$ and this holds for all $r>0$, hence on $W$. Note that $g_{2k}(0)=1$, so, there exists $r_0>0$ such that $|g_{2k}(z)|>0$ on $B(0,2r_0)$ and hence $g_k(z)/g_{2k}(z)$ is analytic function on $W\cup B(0,2r_0)$. Let $C_{r_0}>0$ be a constant such that $\left|g_k(z)/g_{2k}(z)\right|\leq C_{r_0}$ on $\overline{B(0,r_0)}$. Take $r_1>0$ such that $\overline{B(x,2d r_1)}\subseteq\left(\CC^d\backslash\overline{B(0,r_0/16)}\right)\cap W$, for all $x\in W_{\frac{r_0}{4}}\cap\RR^d_x$. Then, for such $x$, from Cauchy integral formula, we have
\beqs
\left|\partial_z^{\alpha}f(x)\right|\leq \frac{\alpha !}{r_1^{|\alpha|}}\sup_{|w_1-x_1|\leq r_1,...,|w_d-x_d|\leq r_1}|f(w)|.
\eeqs
Now, for $w=u+iv\in\CC^d$ such that $|w_j-x_j|\leq r_1$, for all $j=1,...,d$, using the estimate (\ref{30}) but with $2k$ instead of $k$ and the fact $\mathrm{Re\,}\sqrt{z^2}>0$, for $z\in W$, which we proved above, we get
\beqs
|f(w)|&=&\left|\frac{e^{k\sqrt{w^2}}+e^{-k\sqrt{w^2}}}{e^{2k\sqrt{w^2}}+e^{-2k\sqrt{w^2}}}\right|\leq \frac{e^{k\sqrt[4]{\left(|u|^2-|v|^2\right)^2+4(uv)^2}}+1}
{e^{2c_1 k\sqrt{|u|^2-|v|^2}}-1}\\
&\leq&\frac{2e^{k\sqrt{|u|^2-|v|^2+2|uv|}}}
{e^{2c_1 k\sqrt{|u|^2-|v|^2}}-1}\leq \frac{2e^{\sqrt{2}k|u|}}
{e^{\sqrt{3}c_1 k|u|}-1}\leq C_1 e^{(\sqrt{2}-\sqrt{3}c_1)k|u|}
\eeqs
and it is easy to check that $\sqrt{2}-\sqrt{3}c_1<0$. If we put $c=\sqrt{3}c_1-\sqrt{2}$, we get
\beqs
|f(w)|\leq C_1 e^{-ck|u|}\leq C_1 e^{-ck(|x|-|u-x|)}\leq C_1 e^{ck r_1\sqrt{d}}e^{-ck|x|}= C_2 e^{-ck|x|}.
\eeqs
Hence $\ds \left|\partial_x^{\alpha}f(x)\right|\leq C_2\frac{\alpha !}{r_1^{|\alpha|}}e^{-ck|x|}$. For $x\in (B(0,r_0/2)\cap \RR^d_x)\backslash\{0\}$, if we take $r_2>0$ small enough such that $\overline{B(x,2d r_2)}\subseteq B(0,r_0)$ we have (from Cauchy integral formula)
\beqs
\left|\partial_x^{\alpha}f(x)\right|=\left|\partial_z^{\alpha}\left(\frac{g_k(x)}{g_{2k}(x)}\right)\right|\leq \frac{\alpha!}{r_2^{\alpha}}\sup_{|w_1-x_1|\leq r_2,...,|w_d-x_d|\leq r_2}|g(w)|\leq C_{r_0} \frac{\alpha!}{r_2^{\alpha}}\leq C_3 \frac{\alpha!}{r_2^{\alpha}}e^{-ck|x|}.
\eeqs
Because $f(x)$ is in $\mathcal{C}^{\infty}(\RR^d)$ the same inequality will hold for the derivatives in $x=0$. If we take $r=\min\{r_1,r_2\}$ we get that, for $x\in\RR^d$,
\beq\label{35}
\left|\partial_x^{\alpha}f(x)\right|\leq C\frac{\alpha!}{r^{\alpha}}e^{-ck|x|},
\eeq
for some $C>0$. From this it easily follows that $\ds f(x)=\frac{\cosh(k|x|)}{\cosh(2k|x|)}\in\SSS^{*}$.
\end{proof}

\begin{lemma}\label{37}
If $\psi\in\SSS^{*}$ and $T\in\SSS'^{*}$ then $\psi T\in\OO_C'^{*}$.
\end{lemma}
\begin{proof} The Fourier transform is a bijection between $\OO_C'^{*}$ and $\OO_M^{*}$ (see proposition 8 of \cite{PBD}) and $\mathcal{F}(\psi T)=\mathcal{F}\psi*\mathcal{F}T$. Hence, it is enough to prove that $\psi * T\in\OO_M^{*}$ for all $\psi\in\SSS^{*}$ and $T\in\SSS'^{*}$. From the representation theorem of ultradistributions in $\SSS'^*$ (theorem 2 of \cite{PilipovicT}), there exists locally integrable function $F(x)$ (in fact it can be taken to be continuous) such that there exist $m,C>0$, resp. for every $m>0$ there exists $C>0$, such that $\ds\left\|F(x) e^{-M(m|x|)}\right\|_{L^{\infty}}\leq C$ and an ultradifferential operator $P(D)$ of class * such that $T=P(D)F$. Because
\beqs
\psi *T=\psi *P(D)F=P(D)(\psi *F)=P(D)\psi *F
\eeqs
and $P(D)\psi\in\SSS^*$ it is enough to prove that for every $\psi\in\SSS^*$ and every such $F$, $\psi *F\in\OO_M^*$. We will give the proof only in the $\{M_p\}$ case, the $(M_p)$ case is similar. Let $\psi$ and $F$ are such function. There exists $h>0$ such that
\beqs
\sup_{\alpha\in\NN^d}\sup_{x\in\RR^d}\frac{h^{|\alpha|}e^{M(h|x|)}\left|D^{\alpha}\psi(x)\right|}{M_{\alpha}}<\infty.
\eeqs
Take $m$ such that $\ds\int_{\RR^d}e^{-M(h|t|)}e^{M(2m|t|)}dt$ is finite. Later on we will impose another condition on $m$. Then $\ds\left\|F(x) e^{-M(m|x|)}\right\|_{L^{\infty}}\leq C_m$. Note that $e^{M(m|x-t|)}\leq 2e^{M(2m|x|)} e^{M(2m|t|)}$ (one easily proves that for $\lambda,\nu>0$, $e^{M(\lambda+\nu)}\leq 2e^{M(2\lambda)}e^{M(2\nu)}$), so we have
\beqs
\left|D^{\alpha}(\psi *F)(x)\right|&\leq&\int_{\RR^d}\left|D^{\alpha}\psi(t)\right||F(x-t)|dt\leq C'C_m \int_{\RR^d}\frac{e^{-M(h|t|)}M_{\alpha}}{h^{|\alpha|}}e^{M(m|x-t|)}dt\\
&\leq& C'C_m C''\frac{e^{M(2m|x|)}M_{\alpha}}{h^{|\alpha|}}\int_{\RR^d}e^{-M(h|t|)}e^{M(2m|t|)}dt\leq
C\frac{e^{M(2m|x|)}M_{\alpha}}{h^{|\alpha|}}.
\eeqs
We will use the equivalent condition given in proposition 7 of \cite{PBD} for a $\mathcal{C}^{\infty}$ function to be a multiplier for $\SSS'^{\{M_p\}}$. Let $k>0$ be arbitrary but fixed. Take $m$ small enough such that $2m\leq k$. Choose $h_1<h$. Then, by the previous estimates, we obtain
\beqs
\frac{h_1^{|\alpha|} e^{-M(k|x|)}\left|D^{\alpha}(\psi *F)(x)\right|}{M_{\alpha}}\leq
C\frac{h_1^{|\alpha|} e^{-M(k|x|)}e^{M(2m|x|)}M_{\alpha}}{h^{|\alpha|}M_{\alpha}}\leq C,
\eeqs
hence $\psi *F$ is a multiplier for $\SSS'^{\{M_p\}}$ and the proof is complete.
\end{proof}

\indent For $S\in B^{*}$, by lemma \ref{27}, for $k>0$, $\ds\frac{\cosh(k|x|)}{\cosh(2k|x|)}\in\SSS^{*}$ and by lemma \ref{37} we have
\beqs
\cosh(k|x|)S=\frac{\cosh(k|x|)}{\cosh(2k|x|)}\cosh(2k|x|)S\in\OO_C'^{*}.
\eeqs
Similarly as in the proof of lemma \ref{27} one can prove that $\left(\cosh(k|x|)\right)^{-1}\in\SSS^{*}$, for $k>0$. So, for $S\in B^{*}$, we also have $\ds S=\left(\cosh(k|x|)\right)^{-1}\cosh(k|x|)S\in\OO_C'^{*}$. Using this, we get
\beq\label{40}
B^{*}=\left\{S\in\DD'^{*}|\cosh(k|x|) S\in\OO_C'^{*},\, \forall k\geq 0\right\}.
\eeq
\begin{lemma}
$\OO_C'^{(M_p)}\subseteq\DD'^{(M_p)}_{L^1}$ and $\OO_C'^{\{M_p\}}\subseteq \tilde{\DD}'^{\{M_p\}}_{L^1}$.
\end{lemma}
\begin{proof} We will give the proof only in the $\{M_p\}$ case, the $(M_p)$ case is similar. Let $S\in\OO_C'^{\{M_p\}}$. From proposition 2 of \cite{PBD}, there exist $k>0$ and $\{M_p\}$ - ultradifferential operator $P(D)$ such that $S=P(D)F_1+F_2$ where $\left\|e^{M(k|x|)}\left(|F_1(x)|+|F_2(x)|\right)\right\|_{L^{\infty}}<\infty$. We will assume that $F_2=0$ and put $F=F_1$. The general case is proved analogously. Let $\varphi\in\DD^{\{M_p\}}$. We have
\beqs
\left|\langle S,\varphi\rangle\right|=|\langle F,P(-D)\psi\rangle|\leq \left\|e^{M(k|\cdot|)}F\right\|_{L^{\infty}}\left\|e^{-M(k|\cdot|)}\right\|_{L^1}\|P(-D)\varphi\|_{L^{\infty}}\leq Cp_{(t_j)}(\varphi),
\eeqs
for some $C>0$ and $(t_j)\in\mathfrak{R}$, where, the last inequality follows from the fact that $P(D): \dot{\tilde{\mathcal{B}}}^{\{M_p\}}\longrightarrow \dot{\tilde{\mathcal{B}}}^{\{M_p\}}$ is continuous. Because $\DD^{\{M_p\}}$ is dense in $\dot{\tilde{\mathcal{B}}}^{\{M_p\}}$, the claim in the lemma follows.
\end{proof}

\indent If we use the previous lemma in (\ref{40}), we get
\beq
B^{(M_p)}=\left\{S\in\DD'^{(M_p)}|\cosh(k|x|) S\in\DD'^{(M_p)}_{L^1},\, \forall k\geq 0\right\},\label{45}\\
B^{\{M_p\}}=\left\{S\in\DD'^{\{M_p\}}|\cosh(k|x|) S\in\tilde{\DD}'^{\{M_p\}}_{L^1},\, \forall k\geq 0\right\}.\label{46}
\eeq

Now we will give the theorem that characterizes the elements of $\DD'^*$ for which the convolution with $e^{s|x|^2}$ exists as an element of $\DD'^*$.

\begin{theorem}\label{ter}
Let $s\in\RR$, $s\neq 0$. Then
\begin{itemize}
\item[a)] The convolution of $S\in\DD'^{*}$ and $e^{s|x|^2}$ exists if and only if $S\in B^{*}_s$.
\item[b)] $\mathcal{L}: B^{*}\longrightarrow A^{*}$ is well defined and bijective mapping. For $S\in B^{*}$ and $\xi,\eta\in\RR^d$, $e^{-(\xi+i\eta)x}S(x)\in\DD_{L^1}'^{(M_p)}\left(\RR^d_x\right)$, resp. $e^{-(\xi+i\eta)x}S(x)\in\tilde{\DD}_{L^1}'^{\{M_p\}}\left(\RR^d_x\right)$ and the Laplace transform of $S$ is given by $\mathcal{L}(S)(\xi+i\eta)=\left\langle e^{-(\xi+i\eta)x}S(x),1_x\right\rangle$.
\item[c)] The mapping $B^{*}_s\longrightarrow A^{*}_s$, $S\mapsto S* e^{s|x|^2}$ is bijective and for $S\in B^*_s$, $\left(S*e^{s|\cdot|^2}\right)(x)=e^{s|x|^2}\mathcal{L}\left(e^{s|\cdot|^2}S\right)(2sx)$.
\end{itemize}
\end{theorem}
\begin{proof} First we will prove $a)$. Let $S\in B^{*}_s$. Let $\varphi\in\DD^{*}$ is fixed and $K\subset\subset \RR^d$, such that $\mathrm{supp\,}\varphi\subseteq K$. Note that
\beqs
\left(\varphi *e^{s|\cdot|^2}\right)(x)=e^{s|x|^2}\int_{\RR^d}\varphi(y)e^{s|y|^2-2sxy}dy
\eeqs
and define $\ds f(x)=\left(\cosh(k|x|)\right)^{-1}\int_{\RR^d}\varphi(y)e^{s|y|^2-2sxy}dy$ where $k$ will be chosen later. Put $l=\sup\{|y||y\in K\}$ to simplify notations. We will prove that $f\in\DD_{L^{\infty}}^*$, for large enough $k$. For $w\in \CC^d$, put $\ds g(w)=\int_{\RR^d}\varphi(y)e^{s|y|^2-2swy}dy$. Then $g(w)$ is an entire function. To estimate its derivatives we use the Cauchy integral formula and obtain
\beqs
\left|\partial^{\alpha}g(x)\right|\leq \frac{\alpha!}{r^{|\alpha|}}\sup_{|w_1-x_1|\leq r,...,|w_d-x_d|\leq r}|g(w)|.
\eeqs
Take $r<1/(2dl|s|)$. We put $w=\xi+i\eta$ and estimate
\beqs
|g(w)|&\leq& \int_{\RR^d}|\varphi(y)|e^{s|y|^2-2s\xi y}dy\leq e^{2|s||\xi|l}\|\varphi\|_{L^{\infty}}\int_{K}e^{s|y|^2}dy=
c''\|\varphi\|_{L^{\infty}}e^{2l|s||\xi|}\\
&\leq& c''\|\varphi\|_{L^{\infty}}e^{2l|s|\left(|x|+|\xi-x|\right)}= c''\|\varphi\|_{L^{\infty}} e^{2l|s||\xi-x|}e^{2l|s||x|}\leq 3c''\|\varphi\|_{L^{\infty}}e^{2l|s||x|},
\eeqs
where we denote $\ds c''=\int_{K}e^{s|y|^2}dy$. Hence, we get
\beq\label{50}
\left|\partial_x^{\alpha}g(x)\right|\leq \frac{3c''\|\varphi\|_{L^{\infty}}\alpha!}{r^{|\alpha|}}e^{2l|s||x|}.
\eeq
We can use the same methods as in the proof of lemma \ref{27} to prove that
\beqs
\left|D^{\alpha}\left(\frac{1}{\cosh(k|x|)}\right)\right|\leq C\frac{\alpha!}{r^{|\alpha|}}e^{-c'k|x|}
\eeqs
for some $C>0$, $c'>0$ and $c'$ doesn't depend on $k$. If we take $r>0$ small enough we can make it the same for (\ref{50}) and the above estimate. Now take $k$ large enough such that $2l|s|<c'k$. Then, for $h>0$ fixed, we have
\beqs
\frac{h^{|\alpha|}\left|D^{\alpha}f(x)\right|}{M_{\alpha}}&\leq&
\sum_{\beta\leq\alpha}{\alpha\choose\beta}\frac{h^{|\alpha|}\left|D^{\alpha-\beta}g(x)\right|
\left|D^{\beta}\left(\left(\cosh(k|x|)\right)^{-1}\right)\right|}{M_{\alpha}}\\
&\leq&3c'' C\|\varphi\|_{L^{\infty}} \sum_{\beta\leq\alpha}{\alpha\choose\beta}\frac{(2h)^{|\alpha|}(\alpha-\beta)!e^{2l|s||x|}
\beta!e^{-c'k|x|}}{2^{|\alpha|}r^{|\alpha-\beta|}r^{|\beta|}M_{\alpha}}\\
&\leq&\frac{3c'' C\|\varphi\|_{L^{\infty}}}{2^{|\alpha|}} \sum_{\beta\leq\alpha}{\alpha\choose\beta}\left(\frac{2h}{r}\right)^{|\alpha|}
\frac{\alpha!}{{M_{\alpha}}}e^{(2l|s|-c'k)|x|}\leq c''C'\|\varphi\|_{L^{\infty}},
\eeqs
where we use the fact $\ds\frac{k^p p!}{M_p}\rightarrow 0$ when $p\rightarrow\infty$. From the arbitrariness of $h$ we have $f\in\DD_{L^{\infty}}^*$. Because $\DD_{L^{\infty}}^{\{M_p\}}=\tilde{\DD}_{L^{\infty}}^{\{M_p\}}$ as a set, $f\in\tilde{\DD}_{L^{\infty}}^{\{M_p\}}$. Now, we obtain
\beqs
\left(\varphi *e^{s|\cdot|^2}\right)(x)S=f(x)\cosh(k|x|)e^{s|x|^2}S.
\eeqs
$e^{s|x|^2}S\in B^{*}$ (because $S\in B^{*}_s$), hence, by (\ref{45}), resp. (\ref{46}), $\cosh(k|x|)e^{s|x|^2}S\in \DD_{L^1}'^{(M_p)}$, resp. $\cosh(k|x|)e^{s|x|^2}S\in \tilde{\DD}_{L^1}'^{\{M_p\}}$. Hence $\left(\varphi *e^{s|x|^2}\right)S\in\DD_{L^1}'^{(M_p)}$, resp. $\left(\varphi *e^{s|x|^2}\right)S\in\tilde{\DD}_{L^1}'^{\{M_p\}}$. The theorem of \cite{PK} implies that the convolution of $S$ and $e^{s|x|^2}$ exists, in the $(M_p)$ case. Let us consider the $\{M_p\}$ case. If we prove that for arbitrary compact subset $K$ of $\RR^d$, the bilinear mapping $(\varphi,\chi)\mapsto \left\langle \left(\varphi*e^{s|\cdot|^2}\right)S,\chi\right\rangle$, $\DD^{\{M_p\}}_K\times \dot{\tilde{\mathcal{B}}}^{\{M_p\}}\longrightarrow \CC$, is continuous, theorem 1 of \cite{PB} will imply the existence of convolution of $S$ and $e^{s|x|^2}$. Let $K\subset\subset \RR^d$ be fixed. By the above consideration, we have
\beqs
\left|\left\langle \left(\varphi*e^{s|\cdot|^2}\right)S,\chi\right\rangle\right|=\left|\left\langle \cosh(k|x|)e^{s|x|^2}S(x),f(x)\chi(x)\right\rangle\right|\leq C_1p_{(t_j)}(f\chi),
\eeqs
for some $C_1>0$ and $(t_j)\in\mathfrak{R}$, where, in the last inequality, we used that $\cosh(k|x|)e^{s|x|^2}S\in \tilde{\DD}_{L^1}'^{\{M_p\}}$. For brevity, denote $T_{\alpha}=\prod_{j=1}^{|\alpha|}t_j$ and $T_0=1$. Observe that
\beqs
\frac{\left|D^{\alpha}\left(f(x)\chi(x)\right)\right|}{T_{\alpha}M_{\alpha}}&\leq& \sum_{\beta\leq\alpha}{\alpha\choose\beta} \frac{\left|D^{\beta}f(x)\right|\left|D^{\alpha-\beta}\chi(x)\right|}{T_{\beta}M_{\beta}T_{\alpha-\beta}M_{\alpha-\beta}} \leq \tilde{C}c''\|\varphi\|_{L^{\infty}}p_{(t_j/2)}(\chi)\\
&\leq& \tilde{C}c''p_{(t_j/2),K}(\varphi)p_{(t_j/2)}(\chi),
\eeqs
where we used the above estimates for the derivatives of $f$. Note that $c''$ does not depend on $\varphi$, only on $K$. From this, the continuity of the bilinear mapping in consideration follows.\\
\indent For the other direction, let the convolution of $S$ and $e^{s|x|^2}$ exists. Then, by the theorem of \cite{PK}, resp. theorem 1 of \cite{PB}, for every $\varphi\in\DD^{*}$, $\left(\varphi *e^{s|\cdot|^2}\right)S\in\DD_{L^1}'^{(M_p)}$, resp. $\left(\varphi *e^{s|\cdot|^2}\right)S\in\tilde{\DD}_{L^1}'^{\{M_p\}}$. Let $\varphi\in\DD^{*}$, such that $\varphi(y)\geq 0$. Put $U=\{y\in\RR^d|\varphi(y)\neq 0\}$ and $t=\sup\{|y||y\in \mathrm{supp\,}\varphi\}$. Then we have
\beqs
\int_{\RR^d}\varphi(y)e^{s|y|^2-2sxy}dy\geq
c e^{\inf_{y\in U}(-2 sxy)},
\eeqs
where $\ds c=\int_{\RR^d}\varphi(y)e^{s|y|^2}dy$. Let $x_0\in\RR^d$ and $\varepsilon>0$ be fixed. There exists $\varphi\in\DD^{*}$, such that $U\subseteq B(x_0,\varepsilon)$ ($B(x_0,\varepsilon)$ is the ball in $\RR^d$ with center at $x_0$ and radius $\varepsilon$). Then
\beqs
\inf_{y\in U}(-2 sxy)\geq \inf_{y\in B(x_0,\varepsilon)}(-2 sxy)=-2sxx_0+\inf_{y\in B(x_0,\varepsilon)}(-2 sx(y-x_0))\geq-2sxx_0-2\varepsilon|s||x|.
\eeqs
We get
\beqs
\int_{\RR^d}\varphi(y)e^{s|y|^2-2sxy}dy\geq c e^{-2sxx_0-2\varepsilon |s||x|}.
\eeqs
Define $\ds f(x)=e^{-2sxx_0-2\varepsilon |s|\sqrt{1+|x|^2}}\left(\int_{\RR^d}\varphi(y)e^{s|y|^2-2sxy}dy\right)^{-1}$. We will prove that $f\in\DD_{L^{\infty}}^{*}$. $\ds g(w)=\int_{\RR^d}\varphi(y)e^{s|y|^2-2swy}dy$ is an entire function. Put $w=\xi+i\eta$. Then, for $w$ in the strip $\RR^d_{\xi}+i\{\eta\in\RR^d||\eta|<1/(8|s|t)\}$ and $y\in\supp\varphi$, we have $|2s\eta y|\leq 2|s||\eta||y|\leq 1/4<\pi/4$, hence
\beqs
\left|\int_{\RR^d}\varphi(y)e^{s|y|^2-2swy}dy\right|\geq
\left|\int_{\RR^d}\varphi(y)e^{s|y|^2-2s\xi y}\cos(2s\eta y)dy\right|\geq \frac{1}{\sqrt{2}}\int_{\RR^d}\varphi(y)e^{s|y|^2-2s\xi y}dy>0.
\eeqs
Moreover, $e^{-2swx_0-2\varepsilon |s|\sqrt{1+w^2}}$ is analytic on the strip $\RR^d_{\xi}+i\{\eta\in\RR^d||\eta|<1/4\}$, where we take the principal branch of the square root which is single valued and analytic on $\CC\backslash (-\infty,0]$. So, for $r_0=\min\{1/4,1/(8|s|t)\}$, $f(w)$ is analytic on the strip $\RR^d+i\{\eta\in\RR^d||\eta|<r_0\}$. To estimate the derivatives of $f$, we use Cauchy integral formula and obtain
\beq\label{60}
\left|\partial^{\alpha}f(x)\right|\leq \frac{\alpha!}{r^{|\alpha|}}\sup_{|w_1-x_1|\leq r,...,|w_d-x_d|\leq r}|f(w)|,
\eeq
where $r<r_0/(2d)$. Put $\rho=\sqrt{\left(1+|\xi|^2-|\eta|^2\right)^2+4(\xi\eta)^2}$, $\ds\cos\theta= \frac{1+|\xi|^2-|\eta|^2}{\sqrt{\left(1+|\xi|^2-|\eta|^2\right)^2+4(\xi\eta)^2}}$ and $\ds\sin\theta=\frac{2\xi\eta}{\sqrt{\left(1+|\xi|^2-|\eta|^2\right)^2+4(\xi\eta)^2}}$, where $\theta\in (-\pi,\pi)$, from what it follows that $\theta\in(-\pi/2,\pi/2)$ (because $\cos\theta>0$ and $\theta\in (-\pi,\pi)$). Then
\beqs
\mathrm{Re\,}\left(\sqrt{1+w^2}\right)&=&
\mathrm{Re\,}\left(\sqrt{\rho}\left(\cos\frac{\theta}{2}+i\sin\frac{\theta}{2}\right)\right)=
\sqrt{\rho}\cos\frac{\theta}{2}=\sqrt{\rho}\sqrt{\frac{\cos\theta+1}{2}}\\
&=&\frac{\sqrt{\rho\cos\theta+\rho}}{\sqrt{2}}=
\frac{1}{\sqrt{2}}\sqrt{1+|\xi|^2-|\eta|^2+\sqrt{\left(1+|\xi|^2-|\eta|^2\right)^2+4(\xi\eta)^2}}\\
&\geq&\frac{1}{\sqrt{2}}\sqrt{1+|\xi|^2-|\eta|^2+1+|\xi|^2-|\eta|^2}=\sqrt{1+|\xi|^2-|\eta|^2},
\eeqs
where the first equality follows from the fact that we take the principal branch of the square root. We obtain
\beqs
|f(w)|&=&\frac{\left|e^{-2swx_0-2\varepsilon |s|\sqrt{1+w^2}}\right|}{\ds\left|\int_{\RR^d}\varphi(y)e^{s|y|^2-2sw y}dy\right|}\leq\frac{\sqrt{2}e^{-2s\xi x_0}
e^{-2\varepsilon |s|\mathrm{Re\,}\left(\sqrt{1+w^2}\right)}}{\ds\int_{\RR^d}\varphi(y)e^{s|y|^2-2s\xi y}dy}\\
&\leq&\frac{\sqrt{2}e^{-2s\xi x_0}e^{-2\varepsilon |s|\sqrt{1+|\xi|^2-|\eta|^2}}}
{\ds\int_{\RR^d}\varphi(y)e^{s|y|^2-2s\xi y}dy}\leq\frac{\sqrt{2}e^{-2s\xi x_0}e^{-2\varepsilon |s||\xi|}}
{c e^{-2s\xi x_0-2\varepsilon |s||\xi|}}\leq C''_0.
\eeqs
So, from (\ref{60}), we have $\left|\partial_x^{\alpha}f(x)\right|\leq C_0\alpha!/r^{|\alpha|}$, for some $C_0>0$. From this it easily follows that $f\in\DD_{L^{\infty}}^{*}$. Now we have
\beq
e^{-2sxx_0-2\varepsilon |s|\sqrt{1+|x|^2}}e^{s|x|^2}S&=&f(x)\left(\varphi *e^{s|\cdot|^2}\right)(x)S\in\DD_{L^1}'^{(M_p)},
\mbox{ resp.}\label{70}\\
e^{-2sxx_0-2\varepsilon |s|\sqrt{1+|x|^2}}e^{s|x|^2}S&=&f(x)\left(\varphi *e^{s|\cdot|^2}\right)(x)S\in\tilde{\DD}_{L^1}'^{\{M_p\}},\label{71}
\eeq
where we used the fact that $\left(\varphi *e^{s|\cdot|^2}\right)S\in\DD_{L^1}'^{(M_p)}$, resp. $\left(\varphi *e^{s|\cdot|^2}\right)S\in\tilde{\DD}_{L^1}'^{\{M_p\}}$ (which, as noted before, follows from the existence of the convolution of $S$ and $e^{s|x|^2}$) and these hold for every $x_0\in\RR^d$ and every $\varepsilon>0$. Now, put $x_0'=2sx_0$, $x_0''=-2sx_0$ and $\varepsilon'=2|s|\varepsilon$. Then, from (\ref{70}), resp. (\ref{71}), we have
\beqs
e^{-xx_0'-\varepsilon'\sqrt{1+|x|^2}}e^{s|x|^2}S\in\DD_{L^1}'^{(M_p)}&,& e^{xx_0''-\varepsilon'\sqrt{1+|x|^2}}e^{s|x|^2}S\in\DD_{L^1}'^{(M_p)}, \mbox{ resp.}\\
e^{-xx_0'-\varepsilon'\sqrt{1+|x|^2}}e^{s|x|^2}S\in\tilde{\DD}_{L^1}'^{\{M_p\}}&,& e^{xx_0''-\varepsilon'\sqrt{1+|x|^2}}e^{s|x|^2}S\in\tilde{\DD}_{L^1}'^{\{M_p\}}
\eeqs
and from arbitrariness of $x_0$ and $\varepsilon>0$ it follows
\beq\label{80}
\cosh(xx_0)e^{-\varepsilon\sqrt{1+|x|^2}}e^{s|x|^2}S\in\DD_{L^1}'^{(M_p)}, \mbox{ resp. } \cosh(xx_0)e^{-\varepsilon\sqrt{1+|x|^2}}e^{s|x|^2}S\in\tilde{\DD}_{L^1}'^{\{M_p\}}
\eeq
for all $x_0\in\RR^d$ and all $\varepsilon>0$. Let $l>0$. Take $x^{(j)}\in\RR^d$, $j=1,...,d$, to be such that $x^{(j)}_q=0$, for $j\neq q$ and $x^{(j)}_j=ld$. Then
\beq\label{85}
\cosh(l|x|)\leq e^{l|x|}\leq \prod_{j=1}^d e^{l|x_j|}\leq \left(\sum_{j=1}^d \frac{1}{d}e^{l|x_j|}\right)^d\leq
\sum_{j=1}^d e^{ld|x_j|}\leq 2\sum_{j=1}^d \cosh\left(x^{(j)}x\right).
\eeq
We will prove that $\ds \cosh(l|x|)\left(\sum_{j=1}^d \cosh\left(2x^{(j)}x\right)\right)^{-1}\in\DD_{L^{\infty}}^{*}$. The function $\ds\sum_{j=1}^d \cosh (2ld w_j)$ is an entire function of $w=\xi+i\eta$. Moreover, for $w\in U=\RR^d_{\xi}+i\{\eta\in\RR^d||\eta|<1/(4ld^2)\}$, we have
\beqs
\left|\sum_{j=1}^d \cosh (2ld w_j)\right|&=&\frac{1}{2}\left|\sum_{j=1}^d \left(e^{2ld \xi_j}+e^{-2ld \xi_j}\right)
\cos(2ld \eta_j)+i\sum_{j=1}^d \left(e^{2ld \xi_j}-e^{-2ld \xi_j}\right)
\sin(2ld \eta_j)\right|\\
&\geq& \frac{1}{2}\left|\sum_{j=1}^d \left(e^{2ld \xi_j}+e^{-2ld \xi_j}\right)
\cos(2ld \eta_j)\right|\geq\frac{\sqrt{2}}{4}\sum_{j=1}^d \left(e^{2ld \xi_j}+e^{-2ld \xi_j}\right)\\
&\geq&\frac{\sqrt{2}}{4}\sum_{j=1}^d e^{2ld |\xi_j|},
\eeqs
hence
\beq\label{90}
\left|\sum_{j=1}^d \cosh (2ld w_j)\right|\geq\frac{\sqrt{2}}{4}\sum_{j=1}^d e^{2ld |\xi_j|}>0, \mbox{ for all } w=\xi+i\eta\in U.
\eeq
For $\cosh(l|x|)$, we already proved that is the restriction to $\RR^d\backslash \{0\}$ of the function $\cosh(l\sqrt{w^2})$ which is analytic on $W=\{w=\xi+i\eta\in\CC^d||\xi|>2|\eta|\}$ (see the proof of lemma \ref{27}). Hence $\ds \cosh(l\sqrt{w^2})\left(\sum_{j=1}^d \cosh (2ld w_j)\right)^{-1}$ is analytic on $W\cap U$. We will use the same notations that were used in the proof of lemma \ref{27}. Similarly as there, put $\ds g_k(w)=\sum_{n=0}^{\infty}\frac{k^{2n}(w^2)^n}{(2n)!}$. Then $g_k(w)=\left(e^{k\sqrt{w^2}}+e^{-k\sqrt{w^2}}\right)/2$, for $w\in W_r\cap\RR^d_{\xi}$ and from the uniqueness of analytic continuation and arbitrariness of $r>0$ it follows $g_k(w)=\left(e^{k\sqrt{w^2}}+e^{-k\sqrt{w^2}}\right)/2$ on $W$. Fix $0<r_0<1/(8ld^3)$. Then, for $w\in \overline{B(0,r_0)}$, by (\ref{90}), we have $\ds \left|g_l(w)\left(\sum_{j=1}^d \cosh (2ld w_j)\right)^{-1}\right|\leq C_{r_0}$. Take $r_1>0$ such that $\overline{B(x,2d r_1)}\subseteq\left(\CC^d\backslash\overline{B(0,r_0/16)}\right)\cap W\cap U$, for all $x\in W_{\frac{r_0}{4}}\cap\RR^d_x$. For such $x$, we use Cauchy integral formula to estimate
\beqs
\left|\partial^{\alpha}\left(\frac{\cosh(l\sqrt{x^2})}{\sum_{j=1}^d \cosh (2ld x_j)}\right)\right|\leq \frac{\alpha!}{r_1^{\alpha}}\sup_{|w_1-x_1|\leq r_1,...,|w_d-x_d|\leq r_1}\left|\frac{\cosh(l\sqrt{w^2})}{\sum_{j=1}^d \cosh (2ld w_j)}\right|.
\eeqs
Now, using (\ref{90}), we have
\beqs
\left|\frac{\cosh(l\sqrt{w^2})}{\sum_{j=1}^d \cosh (2ld w_j)}\right|&\leq& \frac{2}{\sqrt{2}} \frac{e^{l\mathrm{Re\,}\sqrt{w^2}}+e^{-l\mathrm{Re\,}\sqrt{w^2}}}{\sum_{j=1}^d e^{2ld |\xi_j|}}\leq
\frac{4e^{l\sqrt[4]{\left(|\xi|^2-|\eta|^2\right)^2+4(\xi\eta)^2}}}{\sum_{j=1}^d e^{2ld |\xi_j|}}\\
&\leq&\frac{4e^{l\sqrt{|\xi|^2-|\eta|^2+2|\xi\eta|}}}{\sum_{j=1}^d e^{2ld |\xi_j|}}\leq\frac{4 e^{2l|\xi|}}
{\sum_{j=1}^d e^{2ld |\xi_j|}}\leq\frac{8 \cosh(2l|\xi|)}{\sum_{j=1}^d e^{2ld |\xi_j|}}\leq C',
\eeqs
where the last inequality follows from (\ref{85}). Hence, for $x\in W_{\frac{r_0}{4}}\cap\RR^d_x$ we get
\beqs
\left|\partial^{\alpha}\left(\frac{\cosh(l|x|)}{\sum_{j=1}^d \cosh (2ld x_j)}\right)\right|\leq C'\frac{\alpha!}{r_1^{\alpha}}.
\eeqs
For $x\in (B(0,r_0/2)\cap \RR^d_x)\backslash\{0\}$, if we take $r_2>0$ small enough such that $\overline{B(x,2d r_2)}\subseteq B(0,r_0)$ we have (from Cauchy integral formula)
\beqs
\left|\partial^{\alpha}\left(\frac{\cosh(l\sqrt{x^2})}{\sum_{j=1}^d \cosh (2ld x_j)}\right)\right|&=&\left|\partial^{\alpha}\left(\frac{g_l(x)}{\sum_{j=1}^d \cosh (2ld x_j)}\right)\right|\\
&\leq&\frac{\alpha!}{r_2^{\alpha}}\sup_{|w_1-x_1|\leq r_2,...,|w_d-x_d|\leq r_2}
\left|\frac{g_l(w)}{\sum_{j=1}^d \cosh (2ld w_j)}\right|\leq C_{r_0}\frac{\alpha!}{r_2^{\alpha}}.
\eeqs
Because $\ds \cosh(l|x|)\left(\sum_{j=1}^d \cosh\left(2x^{(j)}x\right)\right)^{-1}$ is in $\mathcal{C}^{\infty}(\RR^d)$ the same inequality will hold for the derivatives in $x=0$. If we take $r=\min\{r_1,r_2\}$ we get that, for $x\in\RR^d$,
\beqs
\left|\partial_x^{\alpha}\left(\frac{\cosh(l|x|)}{\sum_{j=1}^d \cosh\left(2x^{(j)}x\right)}\right)\right|\leq C\frac{\alpha!}{r^{\alpha}}.
\eeqs
Now, it easily follows that $\ds \cosh(l|x|)\left(\sum_{j=1}^d \cosh\left(2x^{(j)}x\right)\right)^{-1}\in\DD_{L^{\infty}}^{*}$. From (\ref{80}), we have
\beq\label{100}
\cosh(l|x|)e^{-\varepsilon\sqrt{1+|x|^2}} e^{s|x|^2}S\in\DD_{L^1}'^{(M_p)}, \mbox{ resp. } \cosh(l|x|)e^{-\varepsilon\sqrt{1+|x|^2}} e^{s|x|^2}S\in\tilde{\DD}_{L^1}'^{\{M_p\}},
\eeq
for every $l>0$ and every $\varepsilon>0$. Let $l>0$ be fixed. By considering the function $e^{\varepsilon\sqrt{1+z^2}}$, which is analytic on the strip $\RR^d+i\{y\in\RR^d||y|<1/4\}$, we obtain the estimates $\ds\left|\partial^{\alpha} e^{\varepsilon\sqrt{1+|x|^2}}\right|\leq \tilde{C}\frac{\alpha !}{\tilde{r}^{|\alpha|}}
e^{2\varepsilon\sqrt{1+|x|^2}}$, for $\tilde{r}<1/(8d)$ and some $\tilde{C}>0$. By this and (\ref{35}), for small enough $r>0$, we have
\beqs
\left|D^{\alpha}\left(\frac{\cosh\left(\frac{l|x|}{2}\right)}{\cosh(l|x|)}e^{\varepsilon\sqrt{1+|x|^2}}\right)\right| &\leq&\sum_{\beta\leq\alpha}{\alpha\choose\beta}
\left|D^{\beta}\left(\frac{\cosh\left(\frac{l|x|}{2}\right)}{\cosh(l|x|)}\right)\right|
\left|D^{\alpha-\beta}e^{\varepsilon\sqrt{1+|x|^2}}\right|\\
&\leq&\sum_{\beta\leq\alpha}{\alpha\choose\beta}C'\frac{\beta!}{r^{|\beta|}}e^{-\frac{c}{2}l|x|}\frac{(\alpha-\beta) !} {r^{|\alpha|-|\beta|}}e^{2\varepsilon\sqrt{1+|x|^2}}\\
&\leq& C'\frac{\alpha!}{r^{|\alpha|}}\sum_{\beta\leq\alpha}{\alpha\choose\beta}e^{-\frac{c}{2}l|x|}e^{2\varepsilon\sqrt{1+|x|^2}}
\leq C''\alpha!\left(\frac{2}{r}\right)^{|\alpha|},
\eeqs
where the last inequality will hold if we take $\varepsilon<cl/4$ and $c$ is the one defined in the proof of lemma \ref{27}. We get that $\ds \frac{\cosh\left(\frac{l}{2}|x|\right)}{\cosh(l|x|)}e^{\varepsilon\sqrt{1+|x|^2}}\in\DD_{L^{\infty}}^{*}$. From this and (\ref{100}) we get $\cosh\left(\frac{l}{2}|x|\right) e^{s|x|^2}S\in\DD_{L^1}'^{(M_p)}$, resp. $\cosh\left(\frac{l}{2}|x|\right) e^{s|x|^2}S\in\tilde{\DD}_{L^1}'^{\{M_p\}}$. From the arbitrariness of $l>0$, we have
\beqs
\cosh(l|x|) e^{s|x|^2}S\in\DD_{L^1}'^{(M_p)}, \mbox{ resp. } \cosh(l|x|) e^{s|x|^2}S\in\tilde{\DD}_{L^1}'^{\{M_p\}}
\eeqs
for all $l>0$. By (\ref{45}), resp. (\ref{46}), we have that $e^{s|x|^2}S\in B^{*}$. Hence $S\in B^{*}_s$.\\
\indent Let us prove $b)$. Let $S\in B^{*}$. Similarly as in the proof of lemma \ref{27}, we can prove that for each fixed $\xi\in\RR^d$ there exists $k_{\xi}>0$ ($k$ depends on $\xi$) such that $\ds\frac{e^{-x\xi}}{\cosh(k_{\xi}|x|)}\in\SSS^{*}\left(\RR^d_x\right)$. Then, for fixed $\xi\in\RR^d$, we have
\beqs
e^{-x\xi}S=\frac{e^{-x\xi}}{\cosh(k_{\xi}|x|)}\cosh(k_{\xi}|x|)S\in\SSS'^*\left(\RR^d_x\right).
\eeqs
Hence, by theorem \ref{t1}, the Laplace transform of $S$ exists and belongs to $A^{*}$. Analogously, for $\varepsilon>0$ and $\xi+i\eta$ fixed, we can find $k>0$ ($k$ depends on $\varepsilon$ and $\xi+i\eta$) such that $\ds\frac{e^{-(\xi+i\eta)x}e^{\varepsilon\sqrt{1+|x|^2}}}{\cosh(k|x|)}\in\SSS^{*}\left(\RR^d_x\right)$. Then
\beqs
e^{-(\xi+i\eta)x}e^{\varepsilon\sqrt{1+|x|^2}}S=\frac{e^{-(\xi+i\eta)x}e^{\varepsilon\sqrt{1+|x|^2}}}{\cosh(k|x|)}
\cosh(k|x|)S\in\DD_{L^1}'^{(M_p)}\left(\RR^d_x\right),
\eeqs
in the $(M_p)$ case and resp. $e^{-(\xi+i\eta)x}e^{\varepsilon\sqrt{1+|x|^2}}S\in \tilde{\DD}_{L^1}'^{\{M_p\}}$ in the $\{M_p\}$ case. By (\ref{25}), we have
\beqs
\mathcal{L}(S)(\xi+i\eta)=\left\langle e^{\varepsilon\sqrt{1+|x|^2}}e^{-(\xi+i\eta)x}S(x),e^{-\varepsilon\sqrt{1+|x|^2}}\right\rangle=
\left\langle e^{-(\xi+i\eta)x}S(x),1_x\right\rangle.
\eeqs
The injectivity is obvious. Let us prove the surjectivity. By theorem \ref{t2}, for $f\in A^{*}$ there exists $T\in\DD'^{*}$ such that $e^{-x\xi}T(x)\in\SSS'^{*}\left(\RR^d_x\right)$, for all $\xi\in \RR^d_{\xi}$ and $\mathcal{L}(T)(\xi+i\eta)=f(\xi+i\eta)$. Because $e^{-x\xi}T(x)\in\SSS'^{*}\left(\RR^d_x\right)$, for all $\xi\in \RR^d$ we obtain that $\cosh(x\xi)T(x)\in \SSS'^*\left(\RR^d_x\right)$ for all $\xi\in\RR^d$. Let $k>0$. By the considerations in the proof of $a)$, if take $x^{(j)}\in\RR^d$, $j=1,...,d$, such that $x^{(j)}_q=0$, for $j\neq q$ and $x^{(j)}_j=kd$, we obtain that $\ds \cosh(k|x|)\left(\sum_{j=1}^d \cosh\left(2x^{(j)}x\right)\right)^{-1}\in\DD_{L^{\infty}}^{*}$. Obviously $\DD_{L^{\infty}}^{*}\subseteq \OO_M^*$. Hence
\beqs
\cosh(k|x|)T(x)=
\cosh(k|x|)\left(\sum_{j=1}^d \cosh\left(2x^{(j)}x\right)\right)^{-1}\sum_{j=1}^d \cosh\left(2x^{(j)}x\right)T(x)\in\SSS'^{*}(\RR^d).
\eeqs
We obtain $T\in B^{*}$ and the surjectivity is proved.\\
\indent Now we will prove $c)$. By $a)$, $S *e^{s|\cdot|^2}$ is well defined for $S\in B^{*}_s$. Let $\psi\in\DD^*$ is such that $0\leq\psi\leq 1$, $\psi(x)=1$ when $|x|\leq 1$ and $\psi(x)=0$ when $|x|>2$. Put $\psi_j(x)=\psi(x/j)$ for $j\in\ZZ_+$. Because the convolution of $S$ and $e^{s|x|^2}$ exists,
\beq\label{120}
\left\langle S* e^{s|\cdot|^2},\varphi\right\rangle=\left\langle \left(\varphi* e^{s|\cdot|^2}\right)S,1\right\rangle=
\lim_{j\rightarrow\infty}\left\langle \left(\varphi* e^{s|\cdot|^2}\right)S,\psi_j\right\rangle,
\eeq
for all $\varphi\in\DD^*$. Fix $j\in\ZZ_+$ and observe that $\left\langle\left(\varphi* e^{s|\cdot|^2}\right)S,\psi_j\right\rangle=\left\langle(\psi_jS)* e^{s|\cdot|^2},\varphi\right\rangle$. Let $l\in\NN$ be so large such that $\supp\psi_j\subseteq\{x\in\RR^d|\psi_l(x)=1\}$. We have\\
$\ds\left\langle \left(\varphi *e^{s|\cdot|^2}\right)S,\psi_j\right\rangle$
\beqs
&=&\left\langle \left(\varphi *e^{s|\cdot|^2}\right)(\xi)(\psi_jS)(\xi),\psi_l(\xi)\right\rangle
=\left\langle e^{s|\xi|^2}\int_{\RR^d}\varphi(x) e^{s|x|^2-2sx\xi}dx(\psi_jS)(\xi),\psi_l(\xi)\right\rangle\\
&=&\left\langle e^{s|\xi|^2}e^{s|x|^2-2sx\xi}(\psi_jS)(\xi),\psi_l(\xi)\varphi(x)\right\rangle
=\left\langle e^{s|x|^2}\left\langle e^{s|\xi|^2}e^{-2sx\xi}(\psi_jS)(\xi),\psi_l(\xi)\right\rangle,
\varphi(x)\right\rangle\\
&=&\left\langle e^{s|x|^2}\left\langle e^{s|\xi|^2}e^{-2sx\xi}S(\xi),\psi_j(\xi)\right\rangle,
\varphi(x)\right\rangle,
\eeqs
where the third and the fourth equality follow from theorem 2.3 of \cite{Komatsu2}. We obtain $\left\langle(\psi_jS)* e^{s|\cdot|^2},\varphi\right\rangle=\left\langle e^{s|x|^2}\left\langle e^{s|\xi|^2}e^{-2sx\xi}S(\xi),\psi_j(\xi)\right\rangle,\varphi(x)\right\rangle$, for all $\varphi\in\DD^*$ and all $j\in\ZZ_+$. Hence
\beq\label{130}
e^{s|x|^2}\left\langle e^{s|\xi|^2}e^{-2sx\xi}S(\xi),\psi_j(\xi)\right\rangle=\left((\psi_jS) *e^{s|\cdot|^2}\right)(x)
\eeq
in $\DD'^*\left(\RR^d_x\right)$, for all $j\in\ZZ_+$. Because $\left\langle e^{s|\xi|^2}e^{-2sx\xi}S(\xi),\psi_j(\xi)\right\rangle=\left\langle \psi_j(\xi)S(\xi),e^{s|\xi|^2}e^{-2sx\xi}\right\rangle$, for every fixed $x\in\RR^d$, theorem 3.10 of \cite{Komatsu3} implies that the left hand side of (\ref{130}) is an element of $\EE^*\left(\RR^d_x\right)$. By (\ref{120}), the right hand side of (\ref{130}) tends to $S*e^{s|\cdot|^2}$ in $\DD'^*$. Because $S\in B^*_s$, $e^{s|\cdot|^2}S\in B^*$ and by $b)$, for each fixed $x,y\in\RR^d$, $e^{-(x+iy)\,\cdot} e^{s|\cdot|^2}S\in\DD_{L^1}'^{(M_p)}$, resp. $e^{-(x+iy)\,\cdot}e^{s|\cdot|^2}S\in\tilde{\DD}_{L^1}'^{\{M_p\}}$, the Laplace transform of $e^{s|\cdot|^2}S$ exists and $\mathcal{L}\left(e^{s|\cdot|^2}S\right)(2sx)=\left\langle e^{s|\xi|^2}e^{-2sx\xi}S(\xi), 1_{\xi}\right\rangle$, for every fixed $x\in\RR^d$. So, the right hand side of (\ref{130}) tends to $e^{s|x|^2}\mathcal{L}\left(e^{s|\cdot|^2}S\right)(2sx)$ pointwise. We will prove that the convergence holds in $\DD'^*$. Let $K$ be a fixed compact subset of $\RR^d$. With similar technic as in the proof of lemma \ref{27}, we can find large enough $k>0$ ($k$ depends on $K$) such that $e^{-2sx\xi}\left(\cosh(k|\xi|)\right)^{-1}\in\SSS^*\left(\RR^d_{\xi}\right)$, for each $x\in K$ and the set $\left\{e^{-2sx\,\cdot}\left(\cosh(k|\cdot|)\right)^{-1}\in\SSS^*\left(\RR^d_{\xi}\right)\big| x\in K\right\}$ is bounded subset of $\SSS^*\left(\RR^d_{\xi}\right)$. Because $S\in B^*_s$, $\cosh(k|\cdot|)e^{s|\cdot|^2}S\in\SSS'^*$. Hence
\beqs
\left\langle e^{s|\xi|^2}e^{-2sx\xi}S(\xi),\psi_j(\xi)\right\rangle&=&\left\langle e^{s|\xi|^2}e^{-2sx\xi}\left(\cosh(k|\xi|)\right)^{-1}\cosh(k|\xi|)S(\xi),\psi_j(\xi)\right\rangle\\
&=&\left\langle e^{s|\xi|^2}\cosh(k|\xi|)S(\xi),e^{-2sx\xi}\left(\cosh(k|\xi|)\right)^{-1}\psi_j(\xi)\right\rangle.
\eeqs
By the way we defined $\psi_j$, one easily verifies that $\left\{e^{-2sx\,\cdot}\left(\cosh(k|\cdot|)\right)^{-1}\psi_j(\cdot)\big|x\in K,\, j\in\ZZ_+\right\}$ is a bounded subset of $\SSS^*\left(\RR^d_{\xi}\right)$. From this it follows that there exists $C_K>0$ ($C_K$ depends on $K$) such that $\left|e^{s|x|^2}\left\langle e^{s|\xi|^2}e^{-2sx\xi}S(\xi),\psi_j(\xi)\right\rangle\right|\leq C_K$, for all $x\in K$, $j\in\ZZ_+$. Because $e^{s|x|^2}\left\langle e^{s|\xi|^2}e^{-2sx\xi}S(\xi),\psi_j(\xi)\right\rangle$ tends to $e^{s|x|^2}\mathcal{L}\left(e^{s|\cdot|^2}S\right)(2sx)$ pointwise, by the above, the convergence also holds in $\DD'^*\left(\RR^d_x\right)$. Hence, we obtain $e^{s|x|^2}\mathcal{L}\left(e^{s|\cdot|^2}S\right)(2sx)= \left(S*e^{s|\cdot|^2}\right)(x)$. Now, $b)$ implies $S *e^{s|\cdot|^2}\in A^{*}_s$. The bijectivity of $S\mapsto S *e^{s|\cdot|^2}$ follows from the bijectivity of $\mathcal{L}:B^{*}\longrightarrow A^{*}$.
\end{proof}

\section{A new class of Anti-Wick operators}

Theorem \ref{ter}, along with (\ref{247}), allows us to define Anti-Wick operators $A_a:\DD^*\left(\RR^d\right)\longrightarrow\DD'^*\left(\RR^d\right)$, when $a$ is not necessary in $\SSS'^*\left(\RR^{2d}\right)$. If $a\in B^*_{-1}$ (and only then) $b(x,\xi)=\pi^{-d}\left(a(\cdot,\cdot)*e^{-|\cdot|^2-|\cdot|^2}\right)(x,\xi)$ exists and is an element of $A^*_{-1}$. If this $b$ is such that, for every $\chi\in\DD^*\left(\RR^{2d}\right)$ the integral
\beq\label{140}
\frac{1}{(2\pi)^d}\int_{\RR^d}\int_{\RR^d}\int_{\RR^d}e^{i(x-y)\xi}b\left(\frac{x+y}{2},\xi\right)\chi(x,y)dxdyd\xi
\eeq
is well defined as oscillatory integral and $\langle K_b,\chi\rangle$ defined as the above integral is well defined ultradistributions, then the operator associated to that kernel (see theorem 2.3 of \cite{Komatsu2}) $\varphi\mapsto \langle K_b(x,y),\varphi(y)\rangle$, $\DD^*\left(\RR^d\right)\longrightarrow\DD'^*\left(\RR^d\right)$, can be called the Anti-Wick operator with symbol $a$ (because of proposition \ref{245}, this is appropriate generalization of Anti-Wick operators). The next theorem gives an example of such $b$.

\begin{theorem}\label{trb}
If $a\in B^*_{-1}$ is such that $b$, given by (\ref{247}), satisfies the following condition: for every $K\subset\subset\RR^d_x$ there exists $\tilde{r}>0$ such that there exist $m,C_1>0$, resp. there exist $C_1>0$ and $(k_p)\in\mathfrak{R}$, (in both cases $C_1$ and $m$, resp. $C_1$ and $(k_p)$ depend on $K$) such that
\beq\label{145}
\left|b(x+i\eta,\xi)\right|\leq C_1 e^{M(m|\xi|)}, \mbox{ resp. } \left|b(x+i\eta,\xi)\right|\leq C_1 e^{N_{k_p}(|\xi|)},\, x\in K, |\eta|<\tilde{r}, \xi\in\RR^d,
\eeq
then (\ref{140}) is oscillatory integral and $K_b$ defined by (\ref{140}) is well defined ultradistribution.
\end{theorem}
\begin{proof} Under the conditions in the theorem, Cauchy integral formula yields $\left|D^{\alpha}_x b(x,\xi)\right|\leq C \alpha!/r_1^{|\alpha|}e^{M(m|\xi|)}$, resp. $\left|D^{\alpha}_x b(x,\xi)\right|\leq C \alpha!/r_1^{|\alpha|}e^{N_{k_p}(|\xi|)}$, for all $x\in K$, $\xi\in\RR^d$ ($r_1$ and $C$ depend on $K$). Let $U$ be an arbitrary bounded open subset of $\RR^{2d}$. Then $V=\left\{t\in\RR^d|t=(x+y)/2,\, (x,y)\in U\right\}$ is a bounded set in $\RR^d$, hence $K=\overline{V}$ is compact set. For this $K$, let $m$, resp. $(k_p)$ be as in (\ref{145}). Take $P_l$, resp. $P_{l_p}$, as in proposition \ref{orn}, such that $\left|P_l(\xi)\right|\geq C_2 e^{M(r|\xi|)}$, resp. $\left|P_{l_p}(\xi)\right|\geq C_2 e^{N_{r_p}(\xi)}$, for some $C_2>0$, such that $\ds\int_{\RR^d}e^{M(m|\xi|)}e^{-M(r|\xi|)}d\xi<\infty$, resp. $\ds\int_{\RR^d}e^{N_{k_p}(|\xi|)}e^{-N_{r_p}(|\xi|)}d\xi<\infty$. We can define $K_{b,U}$ as
\beqs
\langle K_{b,U},\chi\rangle=\frac{1}{(2\pi)^d}\int_{\RR^{3d}}\frac{e^{i(x-y)\xi}}{P_l(\xi)}P_l(D_y) \left(b\left(\frac{x+y}{2},\xi\right)\chi(x,y)\right)dxdyd\xi,\,\chi\in\DD^*(U),
\eeqs
in the $(M_p)$ case, resp. the same but with $P_{l_p}$ in place of $P_l$ in the $\{M_p\}$ case and then one easily checks that $K_{b,U}\in\DD'^*(U)$. Moreover, if $\psi\in \DD^*\left(\RR^d\right)$ is such that $\psi(\xi)=1$ in a neighborhood of $0$, for $\delta>0$, we can define $K_{b,U,\psi,\delta}\in\DD'^*(U)$ as
\beqs
\langle K_{b,U,\psi,\delta},\chi\rangle=\frac{1}{(2\pi)^d}\int_{\RR^{3d}}e^{i(x-y)\xi}\psi(\delta\xi) b\left(\frac{x+y}{2},\xi\right)\chi(x,y)dxdyd\xi.
\eeqs
Then $K_{b,U,\psi,\delta}\longrightarrow K_{b,U}$, when $\delta\longrightarrow 0^+$, in $\DD'^*(U)$. Combining these results, we obtain that the definition of $K_{b,U}$ does not depend on $P_l$ resp. $P_{l_p}$, when these are appropriately chosen (see the above discussion) and on the choice of $\psi$ with the above properties. Moreover, when $U_1$ and $U_2$ are two bounded open sets in $\RR^{2d}$ with nonempty intersection, it follows that $K_{b,U_1}=K_{b,U_1\cup U_1}=K_{b,U_2}$ in $\DD'^*(U_1\cap U_2)$. Because $\DD'^*$ is a sheaf (see \cite{Komatsu1}), $K_b$ can be defined as an element of $\DD'^*\left(\RR^{2d}\right)$ as the oscillatory integral (\ref{140}).
\end{proof}

\begin{example} Interesting such symbols $a$ are given by $e^{l|x|^2}P(\xi)$, where $l<1$ and $P(\xi)$ is an ultrapolynomial of class *. In this case, obviously $a\in B^*_{-1}$. Moreover
\beqs
b(x,\xi)&=&\frac{1}{\pi^d}e^{-|x|^2-|\xi|^2}\mathcal{L}\left(e^{-|\cdot|^2-|\cdot|^2}a(\cdot,\cdot)\right)(-2x,-2\xi)\\
&=& \frac{1}{\pi^d}\left(\frac{\pi}{1-l}\right)^{d/2}e^{l|x|^2/(1-l)}\int_{\RR^d}e^{-|\eta|^2} P(\xi-\eta)d\eta
\eeqs
In the $(M_p)$ case, there exist $m,C_1>0$ such that $|P(\xi-\eta)|\leq C_1 e^{M(m|\xi|)}e^{M(m|\eta|)}$, resp. in the $\{M_p\}$ case, there exist $C_1>0$ and $(k_p)\in\mathfrak{R}$, such that $|P(\xi-\eta)|\leq C_1 e^{N_{k_p}(|\xi|)}e^{N_{k_p}(|\eta|)}$ (in the $(M_p)$ case this estimate follows from proposition 4.5 of \cite{Komatsu1}, in the $\{M_p\}$ case the estimate easily follows by combining proposition 4.5 of \cite{Komatsu1} and lemma 3.4 of \cite{Komatsu3}). Hence, $b$ satisfies the conditions in the above theorem and $b^w$ can be the defined as the operator corresponding to the kernel $K_b$ defined as the oscillatory integral (\ref{140}).
\end{example}



\end{document}